\documentclass{tac}

\usepackage{amsmath,amssymb,amsfonts}
\usepackage{mathrsfs,stmaryrd}
\usepackage{bussproofs}
\usepackage[all]{xy}
\UseComputerModernTips

\author{Marcelo Fiore}

\thanks{
This paper gives a new 
development 
of results announced in~\cite{FioreHur2008} and elaborated upon
in~\cite{Hur2010}.
I am grateful to Chung-Kil Hur for our collaboration on the subject matter
of this work.}

\address{University of Cambridge,\\ Computer Laboratory,\\ 
  15 JJ Thomson Avenue,\\ Cambridge CB3 0FD,\\ UK}

\title{An Equational Metalogic\\
 for Monadic Equational Systems}

\copyrightyear{2013} 


\eaddress{Marcelo.Fiore@cl.cam.ac.uk} 

\newcommand{\cat}[1]{\mathscr{#1}}
\newcommand{\lscat}{\cat}
\newcommand{\monten}{\cdot}
\newcommand{\mh}[2]{\mbox{$#1$-#2}}
\newcommand{\alphaact}{\ul{\alpha}}
\newcommand{\lambdaact}{\ul{\lambda}}
\newcommand{\wt}[1]{\widetilde{#1}}
\newcommand{\enrich}[1]{\ul{#1}}
\newcommand{\ul}[1]{\underline{#1}}
\newcommand{\monact}{\ast}

\newcommand{\monad}[1]{\mathbb{#1}}
\newcommand{\ie}{\textit{i.e.}}

\newcommand{\eg}{\textit{e.g.}}
\newcommand{\cf}{\textit{cf.}}
\newcommand{\tensor}{\otimes}
\newcommand{\St}{\boldsymbol{St}}
\newcommand{\coend}{\int}
\newcommand{\clone}[1]{\langle #1 \rangle}
\newcommand{\adj}{\dashv}
\newcommand{\clsymbol}{C}
\newcommand{\clmon}[1]{\monad\clsymbol_{#1}}
\newcommand{\clfun}[1]{\mathrm{\clsymbol}_{#1}}
\newcommand{\cleta}[1]{\eta^{\clmon{#1}}}
\newcommand{\clmu}[1]{\mu^{\clmon{#1}}}
\newcommand{\st}{\varphi}
\newcommand{\eqnz}[2]{#1 \equiv #2}
\newcommand{\iso}{\cong}
\newcommand{\pair}[1]{\langle #1 \rangle}

\newcommand{\inj}[1]{\iota_{#1}}
\newcommand{\op}{\mathrm{op}}
\newcommand{\name}[1]{\wh{#1}}
\newcommand{\cleval}[2]{\nu^{#1}_{#2}}
\newcommand{\eval}[2]{\varepsilon^{#1}_{#2}}
\newcommand{\eneval}[2]{\ul{\varepsilon}^{#1}_{#2}}
\newcommand{\geneval}[2]{\epsilon^{#1}_{#2}}
\newcommand{\angeneval}[2]{\epsilon'^{#1}_{#2}}
\newcommand{\DDmon}[1]{\monad\DDfunsymbol_{#1}}
\newcommand{\DDeta}[1]{\eta^{\DDmon{#1}}}
\newcommand{\DDmu}[1]{\mu^{\DDmon{#1}}}
\newcommand{\DDmonad}{\DDmon}
\newcommand{\DDfunsymbol}{K}
\newcommand{\DDfun}[1]{\mathrm{\DDfunsymbol}_{#1}}
\newcommand{\DDst}[1]{\kappa^{_{#1}}}
\newcommand{\clst}[1]{\gamma^{_{#1}}}
\newcommand{\DDneg}{\delta}
\newcommand{\stmonmor}{\sigma}
\newcommand{\algstr}{\varsigma}
\newcommand{\intmap}{\iota}
\newcommand{\acthom}[1]{[#1]}
\newcommand{\homact}{\acthom}
\newcommand{\acthomalgstr}[2]{#2_{#1}}
\newcommand{\llrrbrk}[1]{{\llbracket{#1}\rrbracket}}
\newcommand{\Set}{{\boldsymbol{\mathcal S\hspace{-.45mm}\mathit{et}}}}
\newcommand{\Bij}{\mathbb{B}}
\newcommand{\Nat}{\mathbb{N}}
\newcommand{\Op}{{\boldsymbol{\mathcal O\hspace{-.45mm}\mathit p}}}
\newcommand{\Seq}{{\boldsymbol{\mathcal S\hspace{-.45mm}\mathit{eq}}}}
\newcommand{\s}[1]{\mathsf{#1}}
\newcommand{\cotensor}%
{\mathrel{\xy\ar@{}(0,0)|-{\textstyle\cap}\ar@{}(0,0)|-{\textstyle|}\endxy}}
\newcommand{\setof}[1]{\{\, #1 \,\}}
\newcommand{\setofz}[1]{\{ #1 \}}
\newcommand{\catalg}[1]{{#1\text{-}\mathbf{Alg}}}
\newcommand{\wh}[1]{\widehat{#1}}
\newcommand{\subst}[1]{\{ #1 \}}
\newcommand{\eqn}[3]{#1 \vdash #2 \equiv #3}
\newcommand{\ol}[1]{\overline{#1}}

\newcommand{\comp}{\circ}
\newcommand{\id}{\mathrm{id}}
\newcommand{\Id}{\mathrm{Id}}
\newcommand{\catAlg}[2]{{{#1}^{#2}}}
\newcommand{\Kl}[2]{{{#1}_{#2}}}
\newcommand{\Klfun}[1]{{#1}^\star}
\newcommand{\Klst}[1]{\st_{#1}}

\newcommand{\mes}[1]{\mathcal #1}
\newcommand{\Ax}{E}
\newcommand{\myproof}[1]{\mbox{$ #1 \DisplayProof $}}
\newcommand{\tensorext}[2]{\langle{#1}\rangle{#2}}
\newcommand{\arrowlength}{4.25}
\newcommand{\epirightarrow}{\mathrel{\xy\ar@{->>}(\arrowlength,0)\endxy}}
\newcommand{\rightembedding}{\mathrel{\xy\ar@{^(->}(\arrowlength,0)\endxy}}
\newcommand{\qt}{\mathsf{q}}
\newcommand{\epileftarrow}{\mathrel{\xy\ar@{<<-}(\arrowlength,0)\endxy}}

\begin{document}

\maketitle
\begin{abstract}
The paper presents algebraic and logical developments.  
From the algebraic viewpoint, we introduce Monadic Equational Systems as an
abstract enriched notion of equational presentation. 
From the logical viewpoint, we provide Equational Metalogic as a general
formal deductive system for the derivability of equational consequences.
Relating the two, a canonical model theory for Monadic Equational Systems is
given and for it the soundness of Equational Metalogic is established.
This development involves a study of clone and double-dualization
structures.
We also show that 
in the presence of free 
algebras 
the model theory of Monadic Equational Systems satisfies an internal
strong-completeness property.  
\end{abstract}

\tableofcontents

\section{Introduction}
\label{Introduction}

\subsection*{Background}

The modern understanding of equationally defined algebraic
structure,~\ie\ universal algebra, considers the subject as a trinity from the
interrelated viewpoints of: (I)~equational presentations and their varieties;
(II)~algebraic theories and their models; and (III)~monads and their algebras.

The subject was first considered from the viewpoint~(I)
by~\cite{Birkhoff1935}.  
There 
the notion of abstract algebra was introduced and two fundamental results were
proved.  The first one, so-called variety (or HSP) theorem, falls within
the tradition 
of universal algebra and characterises the classes of
equationally defined algebras,~\ie~algebraic categories.  The second one, the
soundness and completeness of equational reasoning, falls within the 
tradition 
of logic and establishes the correspondence between the semantic notion of
validity in all models and the syntactic notion of derivability in a formal
system of inference rules.  

The viewpoints
~(II) and~(III) only became available with the advent of category theory.
Concerning~(II),
~\cite{Lawvere1963} shifted attention from equational presentations to their
invariants in the form of algebraic theories, the categorical counterparts of
the abstract clones of P.\,Hall in universal
algebra~(see~\eg~\cite[Chapter~III, 
page~132]{Cohn1965}).  This opened up a new spectrum of possibilities.  
In particular, 
the notion of algebraic category got extended to that of algebraic functor,
and these were put in correspondence with the concept of map~(or translation)
between algebraic theories.  Furthermore, the 
central 
result that algebraic functors have left adjoints pave the way for the monadic
viewpoint~(III).  In this respect, fundamental results of Linton and of Beck,
see~\eg~\cite[Section~6]{Linton1966}
and~\cite[Section~4]{HylandPower2007}, established the equivalence between
bounded infinitary algebraic theories 
and their set-theoretic models with accessible monads 
on sets and their algebras.
Incidentally, the notion of (co)monad had arisen earlier, in the late 1950s,
in the different algebraic contexts of homological algebra and algebraic
topology (see \eg~\cite[Chapter~VI Notes]{MacLaneCWM}).

\subsection*{Developments}

Since the afore-mentioned original seminal works much has been advanced.  
Specifically, 
the mathematical theories of algebraic theories and monads have been
consolidated and vastly generalised.
Such developments include extensions to categories with structure, to enriched
category theory, and to further notions of algebraic structure.
See, for instance, the developments of~\cite{MacLane1965},
\cite{Ehresmann1968}, \cite{Burroni1971}, \cite{Kelly1972},
\cite{BorceuxDay1980}, \cite{KellyPower1993},\linebreak \cite{Power1999},
\cite{LackPower2009}, \cite{LackRosicky2011}
and the recent accounts in~{\small\cite{AdamekRosicky1994}, \cite{Robinson2002},
  \cite{MacDonaldSobral2004},\linebreak \cite{PedicchioRovatti2004},
  \cite{HylandPower2007}, \cite{AdamekRosickyVitale2010}}.  

By comparison, however, the logical aspect of algebraic theories provided by
equational deduction has been paid less attention to, especially from the
categorical perspective.  An exception 
is the work of~\cite{Rosu2001,
AdamekHebertSousa2007,AdamekSobralSousa2009}.  In these, equational
presentations are abstracted as sets of maps (which in the example of
universal algebra correspond to quotients of free algebras identifying
pairs of terms) and sound and complete deduction systems for the
derivability of morphisms that are injective consequences (which in the
example of universal algebra amount to equational implications) are
considered.

\subsection*{Contribution}

The aim of this work is to contribute to the logical theory of
equationally defined algebraic structure.  Our approach in this 
direction~\cite{FioreHur2008,Hur2010,FioreHur2011} is novel in that it
combines various aspects of the trinity~\mbox{(I--III)}.  

In the first instance, we rely on the concept of 
monad as an abstract notion for describing algebraic structure.  On this
basis, we introduce a general notion of equational presentation, referred
to here as Monadic Equational System~(MES).  This is roughly given by sets
of equations in the form of parallel pairs of Kleisli maps for the monad.
The role played by Kleisli maps here is that of a categorical form of
syntactic term, very much as the role played by the Kleisli category when
distilling a Lawvere theory out of a finitary monad.  

It is of crucial importance, both for applications and theory, that the
categorical development is done in the enriched setting.  In this paper,
as in~\cite{Hur2010} and unlike in the extended
abstract~\cite{FioreHur2008}, we do so from the technically more
elementary and at the same time more general perspective of monoidal
actions, \ie~categories~$\cat C$ equipped with an action~$\monact:\cat
V\times \cat C\to\cat C$ for a monoidal category~$\cat V$, see
Section~\ref{sec:mon-act}.  
In applications, the enrichment is needed, for instance, when moving from mono
to multi sorted algebra, see~\cite[Part~I]{Fiore2008} and~\cite{FioreHur2008}.  
As for the theoretical development, in Section~\ref{sec:MES}, a MES is
then defined to consist of a strong monad~$\monad T$ on a biclosed
action~$(\cat C,{\monact:\cat V\times\cat C\to\cat C})$ for a monoidal
category~$\cat V$ together with a set of equations~$\setof{ \eqnz {u_e}
  {v_e}:C_e\to TA_e }_{e\in E}$ for the endofunctor~$T$ underlying the
monad~$\monad T$ (see Definition~\ref{MonadicEquationalSystem}).  Here the
biclosed structure amounts to right adjoints~$(-)\monact C\dashv
\enrich{\cat C}(C,-):\cat C\to\cat V$ and
$V\monact(-)\dashv\acthom{V,-}:\cat C\to\cat C$ for all $C\in\cat C$ and
$V\in\cat V$.

In Section~\ref{StrongDDmonad}, generalising seminal work
of~\cite{Kock1970a} (see also~\cite{Kock2012}), we show that the biclosed
structure of the monoidal action provides a double-dualization strong
monad~$\DDmon X$ for every $X\in\cat C$, with underlying
endofunctor~$\DDfun X = \acthom{\enrich{\cat C}(-,X),X}$, establishing a
bijective correspondence between \mh{\monad T}algebra structures ${s:TX\to
X}$ and strong monad morphisms~$\stmonmor(s):\monad T\to\monad K_X$.
It follows that Kleisli maps $t:C\rightarrow TA$ have a canonical internal
semantic interpretation in Eilenberg-Moore algebras~$(X,s)$ as morphisms
${\stmonmor(s)_A\comp t:C\to\acthom{\enrich{\cat C}(A,X),X}}$
(\cf~Definition~\ref{KleisliMapInterpretationDefinition} and
Remark~\ref{SemanticsTransformationRemark}), in the same way that in
universal algebra syntactic terms admit algebraic interpretations.  
One thus obtains a canonical notion of satisfaction between algebras and
equations, whereby an equation~${\eqnz uv:C\to TA}$ is satisfied in an
algebra~$(X,s)$ iff its semantic interpretation is an identity, that is
$\stmonmor(s)_A\comp u = \stmonmor(s)_A\comp v : C\to\acthom{\enrich{\cat
C}(A,X),X}$ (see Definition~\ref{SatisfactionRelationDefinition}). 

In Section~\ref{EquationalMetaLogic}, the model theory of MESs is put to
use from the logical perspective, and we introduce a deductive system,
referred to here as Equational Metalogic~(EML), for the formal reasoning
about equations in MESs.
The core of EML are three inference rules---two of congruence and one of
local-character---that embody algebraic properties of the semantic
interpretation.  Hence, EML is sound by design.

In the direction of completeness,
Section~\ref{InternalCompletenessSection} establishes a
strong-completeness result (Theorem~\ref{thm:int-comp}) to the effect
that an equation is satisfied by all models iff it is satisfied by a
freely generated one.  This requires the availability of free
constructions, a framework for which is outlined in
Sections~\ref{FreeStrongMonadsSection} and~\ref{FreeAlgebrasForMESs}.

Strong completeness is the paradigmatic approach to completeness proofs,
and we have in fact already used it to 
this purpose. 
Indeed, the categorical theory of the paper has been shaped not only by
reworking the traditional example of universal algebra in
it~\cite[Part~II]{FioreHur2011} but also by developing two novel
applications.  
Specifically, the companion paper~\cite[Part~II]{FioreHur2011} considers
the framework in the topos of nominal sets~\cite{GabbayPitts2001}, which
is equivalent to the Schanuel
topos~(see~\eg~\cite[page~155]{MacLaneMoerdijk1992}), and studies nominal
algebraic theories providing a sound and complete nominal equational logic
for reasoning about algebraic structure with name-binding operators. 
Furthermore, the companion paper~\cite{FioreHur2010} considers the
framework in the object classifier topos, introducing a conservative
extension of universal algebra from first to second order,~\ie~to
languages with variable binding and parameterised metavariables, and
thereby synthesising a sound and complete second-order equational logic.
Second-order algebraic theories are the subject
of~\cite{FioreMahmoud2010}.

\section{Strong monads}
\label{sec:mon-act}

We briefly review the notion of strong monad (and their morphisms) for an
action of a monoidal category on a 
category (see \eg~\cite{Kock1970b,Kock1972,Pareigis1977}), and recall its
relationship to the notion of enriched monad on an enriched
category (see \eg~\cite{JanelidzeKelly01}).

\subsection*{Monoidal actions}

A \emph{\mh{\cat V}action}~$\cat C = (\cat C, \monact, \alphaact,
\lambdaact)$ for a monoidal category $\cat V=(\cat V, I, \monten,
\alpha,\linebreak
\lambda, \rho)$ consists of a category~$\cat C$, a functor~$\monact: \cat
V \times \cat C \rightarrow \cat C$ and natural
isomorphisms~$\lambdaact_C: I \monact C \stackrel\iso\longrightarrow C$
and $\alphaact_{U,V,C}:(U\monten V)\monact C \stackrel\iso\longrightarrow
U\monact(V\monact C)$ subject to the following coherence conditions:
$$
\xymatrix@C=3pc{
  (I\monten V)\monact C
  \ar[rd]_-{\lambda_V\monact C}
  \ar[r]^-{\alphaact_{I,V,C}}
  &
  I\monact (V\monact C)
  \ar[d]^-{\lambdaact_{V\monact C}}
  \\
  &
  V\monact C
}
\qquad
\xymatrix@C=3pc{
  (V\monten I)\monact C
  \ar[rd]_-{\rho_V\monact C}
  \ar[r]^-{\alphaact_{V,I,C}}
  &
  V\monact (I\monact C)
  \ar[d]^-{V\monact \lambdaact_{C}}
  \\
  &
  V\monact C
}
$$
$$
\xymatrix@C=50pt{
  ((U\monten V)\monten W)\monact C
  \ar[r]^-{\alpha_{U,V,W}\monact C}
  \ar[d]^-{\alphaact_{U\monten V,W,C}}
  &
  (U\monten (V\monten W))\monact C
  \ar[r]^-{\alphaact_{U,V\monten W, C}}
  &
  U\monact ((V\monten W) \monact C)
  \ar[d]^-{U\monact \alphaact_{V,W,C}}
  \\
  (U\monten V)\monact (W \monact C)
  \ar[rr]_-{\alphaact_{U,V,W\monact C}}
  &
  &
  U\monact (V\monact (W \monact C))
}
$$

Such an action is said to be \emph{right closed} if for all $C\in\cat C$
the functor ${(-)\monact C: \cat V \rightarrow \cat C}$ has a right
adjoint ${\enrich{\cat C}(C,-):\cat C \rightarrow \cat V}$ referred to as
a \emph{right-hom}.
The 
action is said to be \emph{left closed} if for all $V\in\cat V$ the
functor ${V\monact (-): \cat C \rightarrow \cat C}$ has a right
adjoint~${\acthom{V,-}:\cat C \rightarrow \cat C}$ referred to as a
\emph{left-hom}.
When an action is both right and left closed, it is said to be
\emph{biclosed}.

\begin{examples}
We will be mainly interested in biclosed actions, examples of which follow.
\begin{enumerate}
\item \label{item:MES-universe-copower}
Every category~$\cat C$ with small coproducts and products gives rise to a
biclosed \mh{\Set}action $(\cat C,\cdot)$, for $\Set$ equipped with the
cartesian structure, where the actions~$V\cdot C$,
right-homs~$\enrich{\cat C}(C,D)$, and left-homs~$\acthom{V,C}$ are
respectively given by the coproducts~$\coprod_{v\in V} C$, the
hom-sets~$\cat{C}(C,D)$, and the products~$\prod_{v\in V} C$.

\item \label{item:MES-universe-monoidal}
Every monoidal biclosed category~$(\cat C,I,\tensor)$ induces the
biclosed $\cat C$-action $(\cat C,\tensor)$ with right-homs and left-homs
respectively given by the right and left closed structures.

\item \label{item:MES-universe-enriched}
For $\cat V$ monoidal closed, every \mh{\cat V}category~$\cat K$ with
tensor~$\tensor$ and cotensor~$\cotensor$ gives rise to the biclosed
$\cat V$-action~$(\cat{K}_0,\tensor_0)$ for $\cat{K}_0$ and $\tensor_0$
respectively the underlying ordinary category and functor of $\cat K$ and
$\tensor$, where the right-homs~$\enrich{\cat{K}_0}(X,Y)$ and
left-homs~$\acthom{V,X}$ are respectively given by the
hom-objects~$\cat{K}(X,Y)$ and the cotensors~$V\cotensor X$.

\item \label{item:MES-universe-product}
From a family of biclosed $\cat V$-actions~${\setof{({\cat
C}_i,\monact_i)}_{i\in I}}$ for a small set~$I$, when $\cat V$ has \mh I
indexed products, we obtain the product biclosed $\cat
V$-action~$\prod_{i\in I} ({\cat C}_i,\monact_i) = (\cat C,\monact)$,
where the category~$\cat C$ is given by the product category~$\prod_{i\in
I}{\cat C}_i$ and where the actions~$V\monact \setofz{C_i}_{i\in I}$,
right-homs~$\enrich{\cat C}\big(\setofz{C_i}_{i\in I},\setofz{D_i}_{i\in
I}\big)$, and left-homs~$\acthom{V,\setofz{C_i}_{i\in I}}$ are
respectively given pointwise by $\setofz{V\monact_i C_i}_{i\in I}$,
$\prod_{i\in I} \enrich{\cat{C}_i}(C_i,D_i)$, and
$\setofz{\acthom{V,C_i}_i}_{i\in I}$.
\end{enumerate}
\end{examples}

\subsection*{Strong functors}

A \emph{strong functor}~$(F,\st): (\cat
C,\monact,\alphaact,\lambdaact)\rightarrow (\cat
C',\monact',\alphaact',\lambdaact')$ between \mh{\cat V}actions consists
of a functor~$F:\cat C \rightarrow \cat C'$ and a \emph{strength}~$\st$
for $F$, \ie~a natural transformation $\st_{V,C}:V\monact' FC
  \rightarrow F(V\monact C): \cat V \times \cat C \rightarrow \cat C$
subject to the following coherence conditions:
$$
\xymatrix@C=3pc{
  I\monact' FC
  \ar[r]^-{\st_{I,C}}
  \ar[rd]_-{\lambdaact'_{FC}}
  &
  F(I\monact C)
  \ar[d]^-{F(\lambdaact_C)}
  \\
  &
  FC
}
\ \ \ 
\xymatrix@C=3pc{
  (U\monten V)\monact' FC
  \ar[r]^-{\alphaact'_{U,V,FC} }
  \ar[d]^-{\st_{U\monten V,C}}
  &
  U\monact' (V \monact' FC)
  \ar[r]^-{U\monact' \st_{V,C}}
  &
  U\monact' F(V\monact C)
  \ar[d]^-{\st_{U,V\monact C} }
  \\
  F((U\monten V)\monact C)
  \ar[rr]_-{F(\alphaact_{U,V,C})}
  &
  &
  F(U\monact (V\monact C))
}
$$

A \emph{strong functor morphism} $\tau:(F,\st)\rightarrow(F',\st')$
between strong functors is a natural transformation $\tau:F\rightarrow F'$
satisfying the coherence condition
$$\xymatrix@C=40pt{
V\monact FC 
\ar[r]^-{V\monact \tau_C}
\ar[d]_-{\st_{V,C}} & 
V\monact F'C 
\ar[d]^-{\st_{V,C}} 
\\
F(V\monact C) 
\ar[r]_-{\tau_{V\monact C}}
& F'(V\monact C)
}$$

\subsection*{Strong monads}

A \emph{strong monad}~$\monad T=(T,\st,\eta,\mu)$ on a \mbox{$\cat
V$-action}~$(\cat C,\monact)$ consists of a strong endofunctor~$(T,\st)$ and a
monad~$(T,\eta,\mu)$ both on $\cat C$ for which the unit~$\eta$ and the
multiplication~$\mu$ are strong functor morphisms $(\Id_{\cat
C},\setof{\id_{V\monact X}}_{V\in\cat V,X\in\cat C})\rightarrow(T,\st)$ and
$(TT,T\st\comp\st T)\rightarrow(T,\st)$; \ie~they satisfy the coherence
conditions below:
$$
\xymatrix@C=3pc{
  V\monact TC
  \ar[r]^-{\st_{V,C}}
  &
  T(V\monact C)
  \\
  V\monact C
  \ar[u]^-{V\monact \eta_C}
  \ar[ru]_-{\eta_{V\monact C}}
}
\qquad
\xymatrix@C=3pc{
  V\monact TTC
  \ar[r]^-{\st_{V,TC}}
  \ar[d]^-{V\monact \mu_C}
  &
  T(V\monact TC)
  \ar[r]^-{T(\st_{V,C})}
  &
  TT(V\monact C)
  \ar[d]^-{\mu_{V\monact C} }
  \\
  V\monact TC
  \ar[rr]_-{\st_{V,C}}
  &
  &
  T(V\monact C)
}
$$

\begin{proposition}
For every strong monad~$\monad T$ on a \mh{\cat V}action $(\cat
C,\monact)$, the \mh{\cat V}action structure on $\cat C$ lifts to a
\mh{\cat V}action structure on the Kleisli category $\Kl{\cat C}{\monad
T}$ making the canonical adjunction 
$\xymatrix{
\cat C
  \ar@<2mm>[r]
  \ar@{}[r]|-{\bot}&\ar@<2mm>[l]
  \Kl{\cat C}{\monad T}
}$ 
into an adjunction of strong functors.
\end{proposition}
The action functor $\monact_{\monad T}:\cat V\times \Kl{\cat C}{\monad
T}\rightarrow \Kl{\cat C}{\monad T}$ is given, for $h:V\rightarrow V'$ in
$\cat V$ and $f:A\rightarrow TA'$ in $\cat C$, by $h\monact_{\monad T}f =
\st_{V',A'}\comp(h\monact f):V\monact A\rightarrow T(V'\monact A')$ in
$\cat C$.  

\bigskip
A \emph{strong monad morphism}
$\tau:(T,\st,\eta,\mu)\rightarrow(T',\st',\eta',\mu')$ between strong
monads on a monoidal action $\cat C$ is a natural transformation
$\tau:T\rightarrow T'$ that is both a strong functor morphism
$\tau:(T,\st)\rightarrow(T',\st')$ and a monad morphism 
$\tau:(T,\eta,\mu)\rightarrow(T',\eta',\mu')$, in that the further
coherence conditions hold:
$$
\xymatrix@C=10pt{
 & \ar[dl]_-{\eta_C} C\ar[dr]^-{\eta'_C} &
 \\
 TC \ar[rr]_-{\tau_C} & & T'C
}
\qquad\qquad\qquad
\xymatrix@C=35pt{
  TTC \ar[r]^-{(\tau\tau)_C} \ar[d]_-{\mu_C} & T'T'C \ar[d]^-{\mu'_C}
  \\
  TC \ar[r]_-{\tau_C} & T'C
}
$$

\begin{proposition}\label{StMonMorToKleisli}
Every morphism $\tau:\monad T\rightarrow\monad T'$ of strong monads on a
\mh{\cat V}action $(\cat C,\monact)$ induces a strong functor
$(\Klfun\tau,\Klst\tau):(\Kl{\cat C}{\monad T},\monact_{\monad
T})\rightarrow(\Kl{\cat C}{\monad T'},\monact_{\monad T'})$ of \mh{\cat
V}actions.
\end{proposition}
\proof
For $f:A\rightarrow TB$ in $\cat C$, $\Klfun\tau(f) = \tau_B\comp
f:A\rightarrow T'B$ in $\cat C$, and $(\Klst\tau)_{V,C}=\id_{V\monact C}$
in $\cat C_{\monad T'}$.
In particular, the diagram
$$\xymatrix@C35pt{
  \ar[d]_-{\Kl\monact{\monad T}}
  \cat V\times \Kl{\cat C}{\monad T} \ar[r]^-{\cat V\times\Klfun\tau} &
  \cat V \times \Kl{\cat C}{\monad T'} 
  \ar[d]^-{\Kl\monact{\monad T'}}
  \\
  \Kl{\cat C}{\monad T}\ar[r]_-{\Klfun\tau} & \Kl{\cat C}{\monad T'}
}$$
commutes.
\endproof

\begin{proposition}
Every morphism $\tau:\monad T\rightarrow\monad T'$ of strong monads on a
monoidal action $\cat C$ contravariantly induces a functor $\catAlg{\cat
C}{\monad T'}\rightarrow\catAlg{\cat C}{\monad
T}:(X,s)\mapsto(X,s\comp\tau_X)$ between the categories of Eilenberg-Moore
algebras.
\end{proposition}

\subsection*{Enrichment}

For a monoidal category~$\cat V$, every right-closed \mh{\cat V}action
induces a \mh{\cat V}category, whose hom-objects are given by the
right-homs.  Furthermore, we have the following correspondences.
\begin{itemize}
\item
  \label{prop:act-enr-fun}
  To give a strong functor between right-closed \mh{\cat V}actions is
  equivalent to give  a \mh{\cat V}functor between the associated
  \mh{\cat V}categories.  

\item
  \label{prop:act-enr-mon}
  To give a strong monad between right-closed \mh{\cat V}actions is
  equivalent to give a \mh{\cat V}monad between the associated \mh{\cat
  V}categories.
\end{itemize}

When $\cat V$ is monoidal \emph{closed}, the notion of right-closed \mh{\cat
V}action essentially amounts to that of tensored \mh{\cat
V}category~(see~{\cite[Section~6]{JanelidzeKelly01}}).  However,
requiring left-closedness for right-closed \mh{\cat V}actions is weaker
than requiring cotensors for the corresponding tensored \mh{\cat
V}categories; as the former requires the action functors~$V\monact(-)$ to
have a right adjoint, whilst the latter further asks that the adjunction
be enriched.  The difference between the two conditions vanishes when
$\cat V$ is \emph{symmetric} monoidal closed.  For example, every monoidal
biclosed category $\cat V$ yields a biclosed \mh{\cat V}action on itself,
but not necessarily a tensored and cotensored \mh{\cat V}category unless
$\cat V$ is symmetric.

\section{Clones and double dualization}
\label{StrongDDmonad}

We consider and study a class of monads that are important in the semantics of
algebraic theories and play a prominent role in the
developments of Sections~\ref{sec:MES}, \ref{EquationalMetaLogic},
and~\ref{InternalCompletenessSection}.  
These monads will be seen to arise from two different constructions,
respectively introduced by~\cite{Kock1970a} for symmetric monoidal closed
categories and by~\cite{KellyPower1993} for locally finitely presentable
categories enriched over symmetric monoidal closed categories that are locally
finitely presentable as closed categories.  Here we generalize these
developments to the setting of biclosed monoidal actions.

Kock's approach sees these monads as arising from a \emph{double-dualization}
adjunction, while Kelly and Power's approach induces them as endo-hom monoids
for a \emph{clone} closed structure.  The latter viewpoint is more general and
allows one to give abstract proofs; hence we introduce it first.  The former
viewpoint is elementary and allows one to apply it more directly.  Both
perspectives complement each other.

\subsection*{Clone monads} 

The constructions of this subsection were motivated by the developments
in~\cite[Sections~4 and~5]{KellyPower1993}.  

\begin{definition}
For \mh{\cat V}actions $\cat A$ and $\cat B$, let $\St(\cat A,\cat B)$ be
the category of strong functors $\cat A\rightarrow\cat B$ and morphisms
between them. 
\end{definition}
Note that the category $\St(\cat A,\cat B)$ is a \mh{\cat V}action with
structure given pointwise.

\begin{theorem}\label{CloneAdjunction}
Let $\cat A$ be a right-closed \mh{\cat V}action and $\cat B$ a
left-closed \mh{\cat V}action.  For every $X\in\cat A$, the evaluation at
$X$ functor $E_X:\St(\cat A,\cat B)\rightarrow \cat B: (F,\st) \mapsto FX$
has the clone functor $\clone{X,-}:\cat B\rightarrow\St(\cat A,\cat B):
Y\mapsto\big(\homact{\enrich{\cat A}(-,X),Y},\clst{X,Y}\big)$ as right
adjoint, where the strength
$\clst{X,Y}_{V,A}:V\monact\clone{X,Y}A\rightarrow\clone{X,Y}(V\monact A)$
is given by the transpose of 
$$\xymatrix@R=20pt{
\enrich{\cat A}(V\monact A,X)\monact(V\monact\acthom{\enrich{\cat A}(A,X),Y})
\ar[d]^-{
\alphaact
^{-1}} 
\\
(\enrich{\cat A}(V\monact A,X)\monten V)
\monact\acthom{\enrich{\cat A}(A,X),Y} 
\ar[d]^-{\geneval{V,A}{X}\monact\homact{\enrich{\cat A}(A,X),Y}}
\\
\enrich{\cat A}(A,X)
\monact
\acthom{\enrich{\cat A}(A,X),Y} 
\ar[d]^-{\eval{\enrich{\cat A}(A,X)}{Y}}
\\
Y
  }$$
with $\geneval{V,A}X:\enrich{\cat A}(V\monact A,X)\monten
V\rightarrow\enrich{\cat A}(A,X)$ in turn the transpose of 
$$\xymatrix@C=15pt{
(\enrich{\cat A}(V\monact A,X)\monten V)\monact A
\ar[rr
]^-{\alphaact
}
&&
\enrich{\cat A}(V\monact A,X)\monact (V\monact A)
\ar[rr]^-{\eneval{V\monact A}{X}}
&&
X
\enspace.
}$$
\end{theorem}
\proof
The main lemmas needed for showing the naturality and coherence conditions
of the strength are as follows:
$$
\begin{array}{c}\xymatrix@C=70pt{
  \enrich{\cat A}(V\monact A,X)\monten U 
  \ar[r]^-{\enrich{\cat A}(h\monact A,X)\monten U}
  \ar[d]_-{\enrich{\cat A}(V\monact A,X)\monten h}
  &
  \enrich{\cat A}(U\monact A,X)\monten U
  \ar[d]^-{\geneval{U,A}X}
  \\
  \enrich{\cat A}(V\monact A,X)\monten V 
  \ar[r]_-{\geneval{V,A}X}
  &
  \enrich{\cat A}(A,X)
}\end{array}
\qquad
(h:U\rightarrow V\mbox{ in }\cat V)
$$
$$\xymatrix{
  & \ar[dl]_-{\enrich{\cat A}(\lambdaact_A,X)\monten I} 
  \enrich{\cat A}(A,X)\monten I \ar[dr]^-{\rho_{\enrich{\cat A}(A,X)}} & 
  \\
  \enrich{\cat A}(I\monact A,X)\monten I \ar[rr]_-{\geneval{I,A}X} & &
  \enrich{\cat A}(A,X)
}$$
$$\xymatrix@C=15pt{
  & 
  \ar[dl]_-{\alpha
}
  (\enrich{\cat A}(U\monact(V\monact A),X)\monten U)\monten V
  \ar[dr]^-{\geneval{U,V\monact A}X\monten V}
  & 
  \\
  \enrich{\cat A}(U\monact(V\monact A),X)\monten(U\monten V)
  \ar[d]_-{\enrich{\cat A}(\alphaact
,X)\monten(U\monten V)}
  &
  &
  \enrich{\cat A}(V\monact A,X)\monten V
  \ar[d]^-{\geneval{V,A}X}
  \\ 
  \enrich{\cat A}((U\monten V)\monact A,X)\monten(U\monten V)
  \ar[rr]_-{\geneval{U\monten V,A}X}
  & & 
  \enrich{\cat A}(A,X)
}$$

The natural bijective correspondence
$$\begin{array}{rcl}
  \algstr^X_{F,Y}:
  &
  \St(\cat A,\cat B)\big(F,\homact{\enrich{\cat A}(-,X),Y}\big)
  \iso
  \cat B(FX,Y)
& : \stmonmor^X_{F,Y}
\end{array}$$
is a form of Yoneda lemma.  Indeed, for a strong functor
morphism~$\tau:F\rightarrow\clone{X,Y}$, one sets 
$$
\algstr(\tau)
=
(\xymatrix{
  FX\ar[r]^-{\tau_X} & 
  \homact{\enrich{\cat A}(X,X),Y}
  \ar[r]^-{\cleval XY}
  & Y
})
$$
where the counit $\cleval XY$ is the composite
$$\xymatrix{
  \homact{\enrich{\cat A}(X,X),Y}
  \ar[rr]^-{\homact{\imath_X,Y}}
  &&
  \homact{I,Y}\ar[r]^-{\lambdaact^{-1}}_-\iso
  & I\monact\homact{I,Y}
  \ar[r]^-{\eval{I}{Y}}_-\iso
  &
  Y
}$$
for $\imath_X$ the transpose of $\lambdaact_X:I\monact X\rightarrow X$,
while for a morphism $f:FX\rightarrow Y$ one lets $\stmonmor(f)$ have
components given by the transpose of 
\begin{equation}\label{InterpretationMap}
\intmap(f)_A
=
(\,\xymatrix@C=15pt{
\enrich{\cat A}(A,X)\monact FA
\ar[rrr]^-{\st_{\enrich{\cat A}(A,X),A}}
&&&
F(\enrich{\cat A}(A,X)\monact A)
\ar[rr]^-{F(\eneval{A}{X})}
&&
FX
\ar[r]^-f
&
Y
})
\enspace.
\end{equation}
\endproof

\begin{corollary}
For a biclosed monoidal action~$\cat C$, the evaluation functor~$\St(\cat
C,\cat C)\times\cat C\rightarrow\cat C$ gives a right-closed monoidal
action structure on $\cat C$ for $\St(\cat C,\cat C)$ equipped with the
composition monoidal structure.
\end{corollary}

Applying the general fact that every object 
of a right-closed \mh{\cat V}action 
canonically induces an endo right-hom monoid 
in $\cat V$ to the situation above, we have that every object of a
biclosed monoidal action $\cat C$ canonically induces a monoid in
$\St(\cat C,\cat C)$, \ie~a strong monad on $\cat C$, and we are lead to
the following.

\begin{definition}
For every object $X$ of a monoidal action $\cat C$, the strong monad
$\clmon X$ on $\cat C$, henceforth referred to as the \emph{clone monad},
has structure given by: 
\begin{itemize}
\item 
  the endofunctor $\clfun X = \clone{X,X}$ with strength $\DDst X=\clst{X,X}$,
\item
  the unit $\cleta X:\Id\rightarrow\clfun X$, and 
\item
  the multiplication $\clmu X:\clfun X\clfun X\rightarrow\clfun X$
\end{itemize}
with the latter two respectively arising as the transposes of
\begin{center}
$\xymatrix{\Id(X)\ar[r]^-{\id_X} & X}$
\enspace and \enspace
$\xymatrix@C=30pt{
  \clfun X\clfun XX \ar[rr]^-{\clfun X\cleval{X}{X}}
  &&
  \clfun XX\ar[r]^-{\cleval{X}{X}} & X
\enspace.
}$
\end{center}
\end{definition}

\subsection*{Double-dualization monads}

For an object~$X$ of a biclosed \mh{\cat V}action~$\cat C$, the monad on
$\cat C$ induced by the adjunction 
\begin{equation}\label{DDadjunction}
  \enrich{\cat C}(-,X)\adj \acthom{-,X}:\cat V^\op\rightarrow\cat C
\end{equation}
will be referred to as the \emph{double-dualization monad}.  
This notion and terminology were introduced by~\cite{Kock1970a} in the
context of symmetric monoidal closed categories.\footnote{The standard
  terminology used in the theoretical computer science literature for
  these monads is \emph{(linear) continuation monads}.
}

\begin{definition}\label{DDMdefinition}
The double-dualization monad $\DDmon X$ on a biclosed monoidal action
$\cat C$ is explicitly given by:
\begin{itemize}
\item 
  the endofunctor $\DDfun X(A) = \acthom{\enrich{\cat C}(A,X),X}$,

\item 
  the unit $\DDeta{X}_A:A\rightarrow K_X(A)$ with components the transpose of
  $\eneval{A}{X}:\enrich{\cat C}(A,X)\monact A\rightarrow X$, and

\item
  the multiplication $$\DDmu{X}_A=\acthom{\DDneg_{\enrich{\cat
  C}(A,X)},X}:K_X(K_X A)\rightarrow K_X(A)$$ where $\DDneg_V:
  V\rightarrow\enrich{\cat C}(\homact{V,X},X)$ is the counit of the
adjunction~(\ref{DDadjunction}) given by the transpose of $\eval{V}{X}:
  V\monact\acthom{V,X}\rightarrow X$.
\end{itemize}
\end{definition}

We observe that, as expected, the clone and double-dualization monads
coincide, from which one has as a by-product that the latter is strong.
\begin{theorem}\label{MonadCoincidence}
For every object $X$ of a biclosed monoidal action $\cat C$, 
$$
\clmon X = \DDmon X
\enspace.
$$
\end{theorem}
\begin{proof}
Since
$\intmap(\id_X)_A=\eval A X:\enrich{\cat C}(A,X)\monact A\rightarrow X$
from~(\ref{InterpretationMap}), 
the units coincide.  
To establish the coincidence of the multiplications, we need show that the
diagram 
$$\xymatrix@C40pt{
\ar[ddd]_-{\DDneg\monact\id}
\enrich{\cat C}(A,X)\monact\clfun X\clfun X A 
\ar[r]^-{\DDst{}} & 
\clfun X(\enrich{\cat C}(A,X)\monact\clfun X A) 
\ar[r]^-{\clfun X(\DDst{})} & 
\clfun X\clfun X(\enrich{\cat C}(A,X)\monact A) 
\ar[d]^-{\clfun X\clfun X(\eneval{}{})}
\\
& &
\clfun X\clfun X X 
\ar[d]^-{\clfun X(\cleval XX)}
\\
& & 
\clfun X X
\ar[d]^-{\cleval XX}
\\
\enrich{\cat C}(\DDfun X A,X)\monact\DDfun X\DDfun X A
\ar[rr]_{\eval{}{}}
& & 
X
}$$
commutes.  This is done using the following fact
$$
\begin{array}{c}
\xymatrix@C=60pt{
\ar[d]_-{\homact{\name f,X}}
\clfun XY
\ar[r]^{\clfun Xf} & 
\clfun XX
\ar[d]^-{\cleval XX}
\\
\homact{I,X} \ar[r]|-\iso & X
}
\end{array}
\qquad 
\mbox{where}
\quad
\begin{array}{c}
  \xymatrix{Y\ar[r]^-f&X}
\\ \hline\hline
  \xymatrix{I\monact Y\ar[r]^-{f\comp\lambdaact_Y}&X}
\\ \hline\hline
\xymatrix{I\ar[r]_-{\name f}&\enrich{\cat C}(Y,X)}
\end{array}
$$
twice, with $f$ being $\eneval AX:\enrich{\cat C}(A,X)\monact A\rightarrow
X$ and $\eval{\enrich{\cat C}(A,X)}{X}:\enrich{\cat C}(A,X)\monact\clfun
X(A)\rightarrow X$, together with the commuting diagrams
$$\xymatrix{
  \ar[d]_-{\eval{}{}}
  \enrich{\cat C}(A,X)\monact\clfun XA \ar[r]^{\DDst{}} & 
  \clfun X(\enrich{\cat C}(A,X)\monact A)
  \ar[d]^-{\homact{\name{\eneval{}{}},X}}
  \\
  X
  \ar[r]|-\iso
  & 
  \homact{I,X}
}$$
and
$$\xymatrix@R15pt{
  \ar[dd]_{\DDneg\monact\id}
  \enrich{\cat C}(A,X)\monact\clfun X\clfun X A
  \ar[r]^-{\DDst{}} & 
  \clfun X(\enrich{\cat C}(A,X)\monact\clfun X A)
  \ar[d]^-{\homact{\name{\eval{}{}},X}}
  \\
  &
  \homact{I,X}
  \ar[d]|-\iso 
  \\
  \ar[r]_-{\eval{}{}}
  \enrich{\cat C}(\DDfun X A,X)\monact\DDfun X\DDfun XA
  & X
}$$
\end{proof}

\subsection*{Algebras}

By a \mh{T}algebra for an endofunctor~$T$ we mean an object $X$ together
with a 
map $TX\rightarrow X$; while a \mh{\monad T}algebra for a monad~$\monad
T$ refers to an Eilenberg-Moore algebra.

\begin{theorem}\label{SemanticsMapTheorem}
For every strong endofunctor $T$ (resp.\ strong monad $\monad T$) on a
biclosed monoidal action $\cat C$, the \mh{T}algebra (resp.\ \mh{\monad
T}algebra) structures on an object $X\in\cat C$ are in bijective
correspondence with the strong endofunctor (resp.\ strong monad) morphisms
$T\rightarrow\clfun X$ (resp.\ $\monad T\rightarrow\clmon X$).  
\end{theorem}
\proof
For endofunctor algebras and strong endofunctor morphisms, the result
follows from Theorem~\ref{CloneAdjunction}, while for monad algebras $s$
and strong monad morphisms $\tau$ one has that $\algstr(\tau)=\cleval
XX\comp\tau_X$ is a \mh{\monad T}algebra because $\cleval XX$ is a
\mh{\clmon X}algebra, and that $\stmonmor(s)$ is a strong monad morphism
because the diagram
$$\xymatrix{
  \ar@/_1em/[ddr]_-{\mu_X} \ar[dr]^-{Ts}
  TTX \ar[rr]^-{\stmonmor(s)_{TX}} 
  && 
  \clfun{X}TX
  \ar[dr]^-{\clfun Xs} \ar[rr]^-{\clfun{X}\stmonmor(s)_X} 
  &&
  \clfun{X}\clfun{X}X 
  \ar@/^1em/[ddl]^-{\clmu X}
  \ar[dl]_-{\clfun{X}\cleval XX}
  \\
  & TX \ar[drr]_-s \ar[rr]^-{\stmonmor(s)_X} 
  && 
  \clfun{X}X \ar[d]_-{\cleval XX} 
  & 
  \\
  & TX \ar[rr]_-s & & X &  
}$$
commutes and because the commutativity of the diagram on the left below
$$
\begin{array}{c}
\xymatrix@R=15pt{
  \ar[ddr]_-{\id_X}
  X\ar[r]^-{\eta_X} & TX \ar[dd]^-{s} \ar[dr]^-{\stmonmor(s)_X}
  \\
  && \clfun{X}X \ar[dl]^-{\cleval X X}
  \\
  & X
}\end{array}
\qquad\qquad\quad
\begin{array}{c}\xymatrix{
  \ar[dr]_-{\cleta X}
  \Id\ar[r]^-{\eta} & T \ar[d]^-{\stmonmor(s)} & 
  \\
  & \clfun{X}
}\end{array}
$$
implies that of the one on the right above.
\endproof

\begin{corollary}\label{DoubleDualizationChange}
Let $T$ (resp.\ $\monad T$) be a strong functor (resp.\ strong monad) on a
\mbox{biclosed} monoidal action $\cat C$.  For every \mh{T}algebra (resp.\
\mh{\monad T}algebra) $(X,s)$ and \mh{\DDfun X}algebra (resp.\
\mh{\DDmon X}algebra) $(Y,k)$, we have
$$\xymatrix{
  & \ar[dl]_-{\stmonmor(s)}
  T 
  \ar[dr]^-{\stmonmor(s_k)}
  & 
  \\
  \DDfun X \ar[rr]_-{\stmonmor(k)} & & \DDfun Y
}$$
where $s_k$ is the \mh{T}algebra (resp.\ \mh{\monad T}algebra)
$$
\xymatrix@C=20pt{
  TY \ar[rr]^-{\stmonmor(s)_Y} && \DDfun X(Y) \ar[r]^-k & Y
  \enspace.
}
$$ 
\end{corollary}

\begin{example}\label{ExponentialDDalgebras}
For every $X\in\cat C$ and $V\in\cat V$, the map
\begin{equation*}\label{DDnegPowerAlgebraStructure}
\homact{\DDneg_V,X}
: \homact{\enrich{\cat C}(\homact{V,X},X),X}
  \rightarrow
  \homact{V,X}
\end{equation*}
provides a \mh{\DDmonad X}algebra structure on $\homact{V,X}$, and we have
the following.
\begin{enumerate}
\item\label{ExponentialDDalgebrasOne}
The associated strong monad morphism
$\stmonmor(\homact{\DDneg_V,X}):
(\DDmonad X,\DDst X)\rightarrow(\DDmonad {\homact{V,X}},\DDst {\homact{V,X}})$ 
has components
$
\homact{\enrich{\cat C}(A,X),X}
\rightarrow
\homact{\enrich{\cat C}(A,\homact{V,X}),\homact{V,X}}
$
given by the double transpose of the composite
$$
\xymatrix{
  V\monact\big(\enrich{\cat C}(A,\homact{V,X})\monact\homact{\enrich{\cat
      C}(A,X),X}\big) 
  \ar[d]^-{
\alphaact
^{-1}}
  \\
  \big(V\monten\enrich{\cat
    C}(A,\homact{V,X})\big)\monact\homact{\enrich{\cat C}(A,X),X}
  \ar[d]^-{\angeneval{V}{A}\monact\homact{\enrich{\cat C}(A,X),X}}
  \\
  \enrich{\cat C}(A,X)\monact\homact{\enrich{\cat C}(A,X),X}
  \ar[d]^-{\eval{\enrich{\cat C}(A,X)}{X}}
  \\
  X
}
$$
where $\angeneval{V}{A}: 
V\monten\enrich{\cat C}(A,\homact{V,X}) \rightarrow \enrich{\cat C}(A,X)$
is in turn the transpose of 
$$
\xymatrix@C=20pt{
  \big(V\monten\enrich{\cat C}(A,\homact{V,X})\big) \monact A
  \ar[r
]^-{\alphaact
}
  &
  V\monact(\enrich{\cat C}(A,\homact{V,X}) \monact A)
  \ar[rr]^-{V\monact\eneval{A}{\homact{V,X}}}
  &&
  V\monact\homact{V,X} 
  \ar[r]^-{\eval{V}{X}}
  & 
  X
}
\ .
$$

\item\label{ExponentialDDalgebrasTwo}
For every strong functor $T$ (resp.\ strong monad $\monad T$) and
\mh{T}algebra (resp.\ \mh{\monad T}algebra) $(X,s)$, the 
\mh{T}algebra (resp.\ \mh{\monad T}algebra) $s_{\homact{\delta_V,X}}$ on
$\homact{V,X}$, for which we will henceforth simply write 
\begin{equation*}\label{sVDefinitionOne}
s_V:T\homact{V,X}\rightarrow\homact{V,X}
\enspace,
\end{equation*}
is the transpose of the composite
\begin{equation*}\label{sVDefinitionTwo}
\xymatrix@C=20pt{
V\monact T\homact{V,X} \ar[rr]^-{\st_{V,\homact{V,X}}} && 
T(V\monact \homact{V,X}) \ar[rr]^-{T(\eval{V}{X})} && 
TX \ar[r]^-{s} & X
\enspace.
}\end{equation*}
\end{enumerate}
\end{example}

\section{Free 
algebras}
\label{FreeStrongMonadsSection}

The category of algebras for an endofunctor is said to admit free algebras
whenever the forgetful functor has a left adjoint.  In this case, the
induced monad is the free monad on the endofunctor.  A wide class of
examples of strong monads arises as such, since the strength of an
endofunctor on a left-closed monoidal action canonically lifts to the
free monad on the endofunctor.  This section establishes a general form of
this result~(Theorem~\ref{thm:general-free-strong-monad}), showing that it
holds for every monad arising from free algebras with respect to full
subcategories of the endofunctor algebras that are closed under left-homs.

\subsection*{Endofunctor algebras}

For an endofunctor~$T$ on a category $\cat C$, the category~$\catalg T$
has \mh{T}algebras as objects and morphisms~$h:(X,s)\rightarrow(Y,t)$
given by maps~$h:X\rightarrow Y$ such that 
$h\comp s = t\comp Th$.
We write $U_T$ for the forgetful functor~${\catalg{T}\rightarrow\cat C}:
(X,s) \mapsto X$.

\begin{definition}
\label{def:act-on-cat-alg}
For a strong endofunctor~$(T,\st)$ on a left-closed \mh{\cat
  V}action~$(\cat C,\monact)$, for every $V\in\cat V$, the left-hom
endofunctor~$\acthom{V,-}$ on $\cat C$ lifts to $\catalg{T}$ by setting 
$$
\big[V,(X,s:TX\rightarrow X)\big] 
\ =\  
\big( \, \acthom{V,X} \, , \, \acthomalgstr V s:T\acthom{V,X}\rightarrow
\acthom{V,X} \, \big)
$$
for $\acthomalgstr V s$ as given in 
Example~\ref{ExponentialDDalgebras}\,(\ref{ExponentialDDalgebrasTwo}).
\end{definition}

For a strong monad $\monad T$ on a left-closed monoidal action $\cat C$,
the left-homs do not only lift to $\catalg T$ but also to the category of
Eilenberg-Moore algebras $\catAlg{\cat C}{\monad T}$.

\begin{lemma}\label{lemma:MES-closed-1}
Let $\monad T$ be a strong monad on a left-closed $\cat
V$-action~$\cat C$.  For every \mh{T}{algebra} $(X,s)$, 
\begin{center}
$
(X,s) \in \catAlg{\cat C}{\monad T}
$
\enspace iff \enspace
$(\acthom{V,X},s_V) \in \catAlg{\cat C}{\monad T}$
for all $V\in\cat V$
\enspace.
\end{center}
\end{lemma}
\proof
($\Rightarrow$) For $(X,s)\in\catAlg{\cat C}{\monad T}$, the equalities
\begin{equation}\label{PowerAlgebraStructure}
\begin{array}{rclcl}
\acthomalgstr V s \comp \eta_{\acthom{V,X}} 
&=&  
\id_{\acthom{V,X}}
&:& 
\acthom{V,X}\rightarrow\acthom{V,X}
\;,
\\[2mm]
\mu_{\acthom{V,X}} \comp \acthomalgstr V s
&=&  
T(\acthomalgstr V s) \comp \acthomalgstr V s
&:& 
TT\acthom{V,X}\rightarrow\acthom{V,X}
\end{array}
\end{equation}
are readily established by considering their transposes.

($\Leftarrow$)
Since the canonical isomorphism~$X\iso\acthom{I,X}$ is a $T$-algebra
isomorphism~$(X,s)\iso(\acthom{I,X},s_I)$, it follows that
$(\acthom{I,X},s_I)\in\cat C^\monad{T}$ implies $(X,s)\in\cat
C^\monad{T}$.
\endproof

\begin{remark}
Under 
the assumption that the action is biclosed,~(\ref{PowerAlgebraStructure})
already follows from Corollary~\ref{DoubleDualizationChange} and
Example~\ref{ExponentialDDalgebras}\,(\ref{ExponentialDDalgebrasTwo}).
\end{remark}

\subsection*{Strong free algebras}

The main result of the section~\cite{FioreHur2008,Hur2010} follows.

\begin{theorem}
\label{thm:general-free-strong-monad}
Let $(F,\st)$ be a strong endofunctor on a left-closed \mh{\cat
V}action~$(\cat C,\monact)$, and consider a full subcategory $\cat A$ of
$\catalg F$ such that the forgetful functor~$\cat A \rightarrow \cat C$ has a
left adjoint, say mapping objects~$X\in \cat C$ to
\mh{F}algebras~${(TX,\tau_X:FTX \rightarrow TX)\in \cat A}$.
\begin{equation}\label{Aadjunction}
\begin{minipage}{\textwidth}
\xymatrix@C=1.5pc@R=2.5pc{
  \cat{A}\,
  \ar@{^(->}[r]
  \ar[rd]
  &
  \catalg{F}
  \ar[d]^-{U_F}
  \\
  &
  \cat{C}
  \ar@<5pt>@/^10pt/[lu]^-{}
  \ar@<6pt>@{}[lu]|(.5)
  {\leftthreetimes}
}
\end{minipage}
\end{equation}
If $\cat A$ is closed under the left-hom endofunctor~$\acthom{V,-}$ for all
$V\in\cat V$, then 
\begin{enumerate}
\item
\label{thm:general-free-strong-monad-1}
for every $(Y,t)\in\cat A$ and map~$f:V\monact X
\rightarrow Y$ in $\cat C$, 
there exists a unique extension map~$f^\#:V\monact TX \rightarrow Y$ in
$\cat C$ such that the diagram 
\begin{equation}\label{eqn:param-ind}
\begin{minipage}{\textwidth}
\xymatrix@C=3pc{
  V \monact FTX
  \ar[d]_-{V\monact \tau_X}
  \ar[r]^-{\st_{V,TX}}
  &
  F (V\monact TX)
  \ar[r]^-{F f^\#}
  &
  FY
  \ar[d]^-{t}
  \\
  V \monact TX
  \ar[rr]^-{\exists !\, f^\#}
  &&
  Y
  \\
  V \monact X
  \ar[u]^-{V\monact \eta_X}
  \ar[rru]_-{f}
}
\end{minipage}
\end{equation}
commutes, and

\item
\label{thm:general-free-strong-monad-2}
the monad~$\monad T=(T,\eta,\mu)$ induced by the
adjunction~(\ref{Aadjunction}) canonically becomes a strong monad, with
the components of the lifted strength~$\wh{\st}$ given by the unique maps
such that the diagram 
\begin{equation}
\label{eqn:monad-str-gen}
\begin{minipage}{\textwidth}
\xymatrix@C=35pt{
V\monact FTC
  \ar[r]^-{\st_{V,TC}}
  \ar[d]_-{V\monact\tau_C}
&
F(V\monact TC)
  \ar[r]^-{F(\wh{\st}_{V,C})}
&
FT(V\monact C)
  \ar[d]^-{\tau_{V\monact C}}
\\
V\monact TC
  \ar@{-->}[rr]^-{\exists{!}\,\wh{\st}_{V,C}}
&&
T(V\monact C)
\\
V\monact C
  \ar[u]^-{V\monact \eta_C}
  \ar[urr]_-{\eta_{V\monact C}}
}
\end{minipage}
\end{equation}
commutes.
\end{enumerate}
\end{theorem}
\proof
(\ref{thm:general-free-strong-monad-1}) For every $F$-algebra~$(Y,t)$ in
$\cat{A}$ also the $F$-algebra~$(\acthom{V,Y},t_V)$ is in $\cat A$.  Thus,
for every map~$f:V\monact X \rightarrow Y$, by the universal property of
the adjunction, there exists a unique extension map~$f^\#:V\monact TX
\rightarrow Y$ making the following diagram commutative
$$
\xymatrix@C=5pc{
  F (TX)
  \ar[r]^-{F (\ol{f^\#})}
  \ar[d]_-{\tau_X}
  &
  F \acthom{V,Y}
  \ar[d]^-{\acthomalgstr V t}
  \\
  TX
  \ar[r]^-{\ol{f^\#}}
  &
  \acthom{V,Y}
  \\
  X
  \ar[u]^-{\eta_X}
  \ar[ru]_-{\ol{f}}
  &
}
$$
where $\ol{f}$ and $\ol{f^\#}$ respectively denote the transposes of the maps
$f$ and $f^\#$.
Transposing this diagram, we obtain diagram~(\ref{eqn:param-ind}) and we
are done.

(\ref{thm:general-free-strong-monad-2})  The above item guarantees the
unique existence of the maps~$\wh{\st}_{V,C}$.  We need show that these
are natural in $V$ and $C$, and satisfy the four coherence conditions of
strengths.

The naturality of $\wh{\st}$,~\ie~that
$
T(f\monact g)\comp \wh{\st}_{V,C} 
= 
\wh{\st}_{V',C'} \comp (f\monact T(g))
$ 
for $f:V\rightarrow V'$ in $\cat V$ and $g:C\rightarrow C'$ in $\cat C$,
is shown by establishing that both these maps are the unique extension of
the composite 
$\xymatrix@C10pt{
  V\monact C \ar[rr]^-{f\monact g} && V'\monact
  C'\ar[rrr]^-{\eta_{V'\monact C'}} &&& T(V'\monact C')
}$.

The first coherence condition
$
T(\lambdaact_C)\comp \wh{\st}_{I,C} = \lambdaact_{TC}
$
is shown by establishing that both these maps are the unique extension of
the composite 
$\xymatrix{
  I\monact C \ar[r]^-{\lambdaact_C} & C \ar[r]^-{\eta_C} & TC
}$.

The second coherence condition
$
T(\alphaact_{U,V,C}) \comp \wh{\st}_{U\monten V, C} 
=
\wh{\st}_{U,V\monact C} \comp (U\monact \wh{\st}_{V,C}) \comp
\alphaact_{U,V,TC}
$
is shown by establishing that both these maps are the unique extension of
the composite
$\xymatrix@C20pt{
  (U\monten V)\monact C \ar[r]^-{\alphaact} & 
  U\monact (V\monact C) \ar[rr]^-{\eta_{U\monact(V\monact C)}} && 
  T(U\monact (V\monact C)) 
}$.

The third coherence condition
$
\wh{\st}_{V,C} \comp (V\monact \eta_C) 
= 
\eta_{V\monact C}
$
is the bottom of diagram~(\ref{eqn:monad-str-gen}).

The last coherence condition
$
\wh{\st}_{V,C} \comp (V\monact \mu_C) 
= 
\mu_{V\monact C} \comp T(\wh{\st}_{V,C}) \comp \wh{\st}_{V,TC}
$
is shown by establishing that both these maps are the unique extension of
$
\wh{\st}_{V,C}: V\monact TC \rightarrow T(V\monact C)
$.
\endproof

\begin{corollary}
For a strong endofunctor $F$ on a left-closed monoidal action $\cat C$ for
which the forgetful functor~$U_F$ has a left adjoint, the induced monad on
$\lscat C$ is strong.
\end{corollary}

\section{Monadic Equational Systems}
\label{sec:MES}

As in~\cite{FioreHur2008,Hur2010}, we introduce a general
abstract enriched notion of equational presentation.  This is here
referred to as Monadic Equational
System~(Definition~\ref{MonadicEquationalSystem}), with the terminology
chosen to indicate the central role played by the concept of monad, which
is to be regarded as encapsulating algebraic structure.  In this context,
equations are specified by pairs of Kleisli maps.

\begin{definition}
\label{Terms}
A \emph{Kleisli map} for an endofunctor~$T$ on a category~$\cat C$ of
arity~$A$ and coarity~$C$ is a morphism~$C\rightarrow TA$ in $\cat C$.  
\end{definition}

\begin{definition}\label{KleisliMapInterpretationDefinition}
For a strong endofunctor $(T,\st)$ on a right-closed \mbox{$\cat
V$-action}~$(\cat C,\monact)$, the \emph{interpretation} of a Kleisli
map~${t:C\rightarrow TA}$ in $\cat C$ with respect to a $T$-algebra~$(X,s)$ is
defined as
\begin{equation}\label{KleisliMapInterpretation}
\llrrbrk{t}_{(X,s)}
= 
\intmap(s)_A \comp (\enrich{\cat C}(A,X)\monact t)
  : \enrich{\cat C}(A,X)\monact C\rightarrow X
\end{equation}
where the \emph{interpretation map}
$\intmap(s)_A: \enrich{\cat C}(A,X)\monact TA\rightarrow X$ is that
defined in~(\ref{InterpretationMap}).
\end{definition}

Two basic properties of interpretation maps follow.
\begin{proposition}\label{SatisfactionTransferProposition}
Let $T$ be a strong endofunctor on a right-closed \mbox{$\cat
V$-action}~$(\cat C,\monact)$.  For $h:(X,s)\rightarrow(Y,t)$ in $\catalg
T$, $h\comp\intmap(s)_A = \intmap(t)_A \comp \big(\enrich{\cat
  C}(A,h)\monact TA\big)$.
\end{proposition}
\begin{proposition}\label{StMonMorInt}
Let $\tau:S\rightarrow T$ be a morphism between strong endofunctors on a
right-closed \mh{\cat V}action $(\cat C,\monact)$.  For every
\mh{$T$}algebra $(X,s)$, the interpretation map
${\intmap(s\comp\tau_X)_A}: \enrich{\cat C}(A,X)\monact SA\rightarrow X$
factors as the composite $\intmap(s)_A\comp(\enrich{\cat
C}(A,X)\monact\tau_A)$.  
\end{proposition}

\begin{remark}\label{SemanticsTransformationRemark}
When considering a biclosed \mh{\cat V}action $\cat C$, the
interpretation maps 
$$
\intmap(s)_A
: \enrich{\cat C}(A,X)\monact TA\rightarrow X
$$ 
transpose to yield a \emph{semantics transformation} 
\begin{equation}\label{SemanticsTransformation}
\stmonmor(s): T \rightarrow \DDfun X
\end{equation}
as introduced in Theorem~\ref{CloneAdjunction} and also studied in
Theorem~\ref{SemanticsMapTheorem}.  
\end{remark}

The interpretation of Kleisli maps in algebras induces a
\emph{satisfaction
  relation}~(Definition~\ref{SatisfactionRelationDefinition}) between
algebras and equations.

\begin{definition}
For an endofunctor $T$, a parallel pair~$\eqnz{u}{v}:C\rightarrow T A$ of
Kleisli maps is referred to as a \emph{\mh{T}equation}.
\end{definition}

\begin{definition}\label{SatisfactionRelationDefinition}
Let $T$ be a strong endofunctor on a right-closed \mh{\cat V}action $(\cat
C,\monact)$.
For all \mh{T}{algebras} $(X,s)$ and \mh{T}equations~$\eqnz u v:
C\rightarrow TA$, \begin{equation*}\label{SatisfactionRelation}
(X,s)\models \eqnz u v: C\rightarrow TA
\enspace \text{iff}\, \enspace 
\llrrbrk u_{(X,s)} = \llrrbrk v_{(X,s)}
 : \enrich{\cat C}(A,X)\monact C\rightarrow X
\enspace.
\end{equation*}
\end{definition}
More generally, for a set of $T$-algebras $\cat{A}$, we set ${\cat{A} \models
\eqnz u v}$ iff ${(X,s)\models \eqnz u v}$ for all ${(X,s)\in \cat{A}}$.

\begin{corollary}\label{SubalgebraClosureCorollary}
Let $T$ be a strong endofunctor on a right-closed \mh{\cat V}action $\cat
C$.  For every $h:(X,s)\rightarrow(Y,t)$ in $\catalg T$ with
$h:X\rightarrow Y$ a monomorphism in $\cat C$, if 
${(Y,t)\models \eqnz u v}$
then $(X,s)\models \eqnz u v$.
\end{corollary}
\proof
By Proposition~\ref{SatisfactionTransferProposition}.  
\endproof

\begin{corollary}\label{corollary:MES-closed-2}
Let $T$ be a strong functor on a biclosed \mh{\cat V}action.  For every
\mh{T}{algebra} $(X,s)$, 
\begin{center}
${(X,s)\models \eqnz u v}$
\enspace iff \enspace
${(\acthom{V,X},s_V) \models \eqnz u v}$
for all $V\in\cat V$
\enspace.
\end{center}
\end{corollary}
\proof
($\Rightarrow$) Because, by 
Corollary~\ref{DoubleDualizationChange} and
Example~\ref{ExponentialDDalgebras}, one has that $\llrrbrk
t_{(\homact{V,X},s_V)}$ is the transpose of the composite
$$\xymatrix@C=15pt{
V\monact(\enrich{\cat C}(A,\homact{V,X})\monact C)
\ar[rr]^-{\alphaact^{-1}}
&&
(V\monten\enrich{\cat C}(A,\homact{V,X}))\monact C
\ar[rr]^-{\angeneval{V}{A}\monact C}
&&
\enrich{\cat C}(A,X)\monact C
\ar[rr]^-{\llrrbrk t_{(X,s)}}
&&
X
}$$
for all $t:C\rightarrow TA$.

($\Leftarrow$) By Corollary~\ref{SubalgebraClosureCorollary} using that
the canonical isomorphism~$X \iso \acthom{I,X}$ is a $T$-algebra
isomorphism~$(X,s)\iso(\acthom{I,X},s_I)$.
\endproof

\subsection*{Monadic Equational Systems}

The idea behind the definition of Monadic Equational System~(MES) is that
of providing a $\cat V$-enriched universe of discourse~$\cat C$ together
with algebraic structure~$\monad T$ for specifying equational
presentations~$E$.  

\begin{definition}
\label{MonadicEquationalSystem}
A \emph{Monadic Equational System} 
$
(\cat V,\cat C,\monad{T},E)
$ 
consists of 
\begin{itemize}
\item 
  a monoidal category~$\cat V = (\cat V,\monten,I,\alpha,\lambda,\rho)$,

\item 
  a biclosed \mh{\cat V}action~$\cat C = \big(\cat
  C,\monact,\alphaact,\lambdaact,\enrich{\cat C}(-,=),\acthom{-,=}\big)$,

\item 
  a strong monad~$\monad T = (T,\st,\eta,\mu)$ on 
    $\cat C$, and

\item a set of \mh{T}equations~$E$.
\end{itemize}
\end{definition}

\begin{remark}\label{SemanticsKleisliFunctorRemark}
Let $\monad T$ be a strong monad on a biclosed \mh{\cat V}action $(\cat
C,\monact)$.  For a \mh{\monad T}algebra $(X,s)$, by
Theorem~\ref{SemanticsMapTheorem}, the semantics
transformation~(\ref{SemanticsTransformation}) is a strong monad morphism
$$
\stmonmor(s): \monad T \rightarrow \DDmon X
$$
that, by Proposition~\ref{StMonMorToKleisli}, induces the
following situation
\begin{equation}\label{SemanticsKleisliFunctor}
\begin{array}{c}\xymatrix@C=60pt{
  \cat V \times \Kl{\cat C}{\monad T}
  \ar[r]^-{\cat V\times \Klfun{\stmonmor(s)}}
  \ar[d]_-{\Kl\monact{\monad T}}
  & 
  \cat V \times \Kl{\cat C}{\DDmon X}
  \ar[d]^-{\Kl\monact{\DDmon X}}
  \\ 
  \cat C_{\monad T}
  \ar[r]_-{\Klfun{\stmonmor(s)}}
  &
  \cat C_{\DDmon X}
}\end{array} 
\end{equation}
where the functorial action of 
${\Klfun{\stmonmor(s)}}
 : \Kl{\cat C}{\monad T}(C,A)\rightarrow\Kl{\cat C}{\DDmon X}(C,A)$ 
is the transpose of the interpretation function of Kleisli
maps~(\ref{KleisliMapInterpretation}).
\end{remark}

\begin{definition}
An \emph{\mh{\mes{S}} algebra} for a MES~$\mes{S} =(\cat V,\cat{C},\monad
T,E)$ is a $\monad T$-algebra~$(X,s)$ satisfying the equations in $E$,
\ie~such that $(X,s) \models \eqnz u v$ for all $(\eqnz u v)\in E$ or,
equivalently, such that $\stmonmor(s)$ coequalizes every parallel pair of
Kleisli maps in $E$.
\end{definition}

The full subcategory of~$\catAlg{\cat C}{\monad T}$ consisting of the
$\mes{S}$-algebras is denoted $\catalg{\mes{S}}$, and we write $U_{\mes
S}$ for the forgetful functor $\catalg{\mes{S}}\rightarrow\cat C$. 

\begin{examples}\label{MESsExamples}
\begin{enumerate}
\item
Every set of $T$-equations~$E$ for a monad~$\monad T$ on a category $\cat C$
with small coproducts and products yields a MES~$(\Set,\cat C,\monad T, E)$.
In particular, bounded infinitary algebraic presentations, see
\eg~\cite{Slominski1959,Wraith1975}, yield such MESs on complete and
cocomplete categories.

\item
An \emph{enriched algebraic theory}~\cite{KellyPower1993} consists of: a
locally finitely presentable category $\cat K$ enriched over a symmetric
monoidal closed category $\cat V$ that is locally finitely presentable as a
closed category together with a small set $\cat K_f$ representing the
isomorphism classes of the finitely presentable objects of $\cat K$; a
\mh{\cat K_f}indexed family of \mh{\cat K}objects $O = \setof{O_c}_{c\in \cat
  K_f}$; and a \mh{\cat K_f}indexed family of parallel pairs of \mh{\cat
  K_0}morphisms $\cat E = \setof{ \eqnz{u_c}{v_c}: E_c \rightarrow
T_O(c)}_{c\in\cat K_f}$ for $\monad T_O$ the free finitary monad on the
endofunctor $\coprod_{c\in\cat K_f} \cat K(c,-)\tensor O_c$ on $\cat K$.  

The structure $(\cat V,\cat K_0,\monad T_O, \cat E)$ yields a MES, an algebra
for which is a \mh{\monad T_O}algebra $(X,s)$ such that
$\stmonmor(s)_c:T_O(c)\rightarrow K_X(c)$ coequalizes $u_c$ and $v_c$ for all
$c\in\cat K_f$.  This coincides with the notion of algebra for the finitary
monad presented by $\cat E$ (by means of a coequaliser of a parallel pair
$\monad T_E\rightrightarrows\monad T_O$ induced by the parallel pairs in $\cat
E$) as discussed in~\cite[Section~5]{KellyPower1993}.

Nominal equational systems are MESs of this kind on the topos of nominal
sets (equivalently the Schanuel topos) that feature
in~\cite[Section~5]{FioreHur2011}.

\item
We exemplify how MESs may be used to provide presentations of algebraic
structure on symmetric operads.  For this purpose, we need consider the
category of symmetric sequences~$\Seq=\Set^\Bij$, for $\Bij$ the groupoid
of finite cardinals and bijections, together with its product and
coproduct structures and the following two monoidal structures:
\begin{itemize}
\item 
  Day's convolution symmetric monoidal closed
  structure~\cite{Day1970},\newline
  \cite{ImKelly1986}
  given by 
  $$\textstyle
  (X\tensor Y)(n) = \coend^{n_1,n_2\in\Bij} X(n_1)\times
  Y(n_2)\times\Bij(n_1+n_2,n)
  $$ 
  with unit $I=\Bij(0,-)$; and 

\item
  the substitution (or composition) monoidal
  structure~\cite{Kelly1972,Joyal1981,FioreGambinoHylandWinskel2008} given by 
  $$\textstyle
  (X\bullet Y)(n) = \coend^{k\in\Bij} X(k) \times Y^{\tensor k}(n)
  $$
  with unit $J=\Bij(1,-)$.
\end{itemize}%

We identify the category of symmetric operads $\Op$ with its well-known
description as the category of monoids for the substitution tensor
product, and proceed to consider algebraic structure on it.  In doing so,
one crucially needs to require that the algebraic and monoid structures
are compatible with each other, see~\cite{FiorePlotkinTuri1999,Fiore2008}.
For example, the consideration of symmetric operads with a cartesian
binary operation~$+$ and a linear binary operation~$*$ leads to defining
the category $\Op\mbox{\small$(+,*)$}$ with objects $A\in\Seq$ equipped
with 
\begin{itemize}
\item 
  a monoid structure $\nu: J\to A$, $\mu: A^{\bullet2}\to A$, and 
  
\item 
  an algebra structure $+: A^2\to A$, $*:A^{\tensor 2}\to A$
\end{itemize}%
that are compatible in the sense that the diagrams 
$$
\xymatrix@C=30pt{
  \ar[d]_-{+\,\bullet\,\id}
  A^{2} \bullet A \ar[rr]^-{\pair{\pi_1\bullet\id,\pi_2\bullet\id}} && 
  (A\bullet A)^{2} \ar[r]^-{\mu^{2}} & A^{2} \ar[d]^-{+}
  \\
  A\bullet A\ar[rrr]_-{\mu} && & A
}
\qquad
\xymatrix{
  \ar[d]_-{*\,\bullet\,\id}
  A^{\tensor 2} \bullet A \ar[r]^-\iso & (A\bullet A)^{\tensor 2}
  \ar[r]^-{\mu^{\tensor 2}} & A^{\tensor 2} \ar[d]^-{*}
  \\
  A\bullet A\ar[rr]_-{\mu} & & A
}
$$
commute.  (Morphisms are both monoid and algebra homomorphisms.)
Then, as follows from the general treatment given in~\cite{Fiore2008},
the forgetful functor $\Op\mbox{\small$(+,*)$}\rightarrow\Seq$ has a left
adjoint, for which the induced monad on $\Seq$ will be denoted $\monad M$.

Algebraic laws correspond to \mh{M}equations, and give rise to
MESs~$(\Set,\Seq,\monad M,E)$.
For example, the left-linearity law 
$$
(x_1+x_2)*x_3 \ = \ x_1*x_3+x_2*x_3
$$ 
corresponds to the \mh{M}equation 
$$\begin{array}{l}
\quad
J^{\tensor2}
\xymatrix@C=75pt
  {\ar[r]^-{\pair{\eta\,\inj1,\eta\,\inj2}\,\tensor\,\eta\,\inj3}&}
  \big(M(3\!\cdot\!J)\big)^2 \tensor M(3\!\cdot\!J) 
\xymatrix@C=35pt{\ar[r]^-{+\,\tensor\,\id}&}
\big(M(3\!\cdot\!J)\big)^{\tensor2}
\xymatrix@C=20pt{\ar[r]^-{*}&}
M(3\!\cdot\!J)
\\[1mm]
\equiv
\\[1mm]
\quad
J^{\tensor2} 
\xymatrix@C=90pt{\ar[r]^-{\pair{
      \eta\,\inj1\tensor\,\eta\,\inj3
      ,
      \eta\,\inj2\tensor\,\eta\,\inj3}}&}
\Big(\big(M(3\!\cdot\!J)\big)^{\tensor2}\Big)^2
\xymatrix@C=20pt{\ar[r]^-{*^2}&}
\big(M(3\!\cdot\!J)\big)^2
\xymatrix@C=20pt{\ar[r]^-{+}&}
M(3\!\cdot\!J)
\end{array}$$
while the additive pre-Lie law 
$$
(x_1*x_2)*x_3+x_1*(x_3*x_2)
\ = \ 
x_1*(x_2*x_3)+(x_1*x_3)*x_2 
$$ 
corresponds to the \mh{M}equation 
$$\begin{array}{l}
\!\!\!\!\!\!
  J^{\tensor3}
  \xymatrix@C=112.5pt{\ar[r]^-{\pair{
        \eta\inj1\tensor\eta\inj2\tensor\eta\inj3
        ,
        \eta\inj1\tensor\eta\inj3\tensor\eta\inj2
      }}&}
  \Big( \big( M(3\!\cdot\!J) \big)^{\tensor3} \Big)^2
  \xymatrix@C=80pt{\ar[r]^-{(*(*\tensor\id))\times(*(\id\tensor*))}&}
  \big( M(3\!\cdot\!J) \big)^2
  \xymatrix@C=15pt{\ar[r]^-+&}
  M(3\!\cdot\!J)
\\[1mm]
\!\!\!\!\!\!
\equiv
\\[1mm]
\!\!\!\!\!\!
  J^{\tensor3}
  \xymatrix@C=112.5pt{\ar[r]^-{\pair{
        \eta\inj1\tensor\eta\inj2\tensor\eta\inj3
        ,
        \eta\inj1\tensor\eta\inj3\tensor\eta\inj2
      }}&}
  \Big( \big( M(3\!\cdot\!J) \big)^{\tensor3} \Big)^2
  \xymatrix@C=80pt{\ar[r]^-{(*(\id\tensor*))\times(*(*\tensor\id))}&}
  \big( M(3\!\cdot\!J) \big)^2
  \xymatrix@C=15pt{\ar[r]^-+&}
  M(3\!\cdot\!J)
\end{array}
$$%
This can 
in fact 
be extended to a MES whose algebras are symmetric operads over vector
spaces equipped with a pre-Lie operation.

The MES framework allows however for greater generality, being able to further
incorporate linear algebraic theories with variable binding
operators~\cite{Tanaka2000} and/or with parameterised
metavariables~\cite{Hamana2004,Fiore2008}.  Details 
may 
appear elsewhere.  Here, as a simple application of the latter, we limit
ourselves to show that one can exhibit an equation 
\begin{center}
\hfill
$ 
\eqnz 
{ u_{m,n} }
{ v_{n,m} }
: J^{\tensor(m\cdot n)} \to T\big(J^{\tensor m}+J^{\tensor n}\big)
$
\hfill
$(m,n\in\Nat)$
\end{center}
for $\monad T$ the monad on $\Seq$ induced by the left adjoint to the
forgetful functor~${\Op\to\Seq}$, that is satisfied by a symmetric operad
iff every two operations respectively of arities $m$ and $n$ commute with
each other.  
Indeed, one lets
$$\begin{array}{rcl}
u_{m,n}
& = &
J^{\tensor(m\cdot n)}
\iso
J^{\tensor m}\bullet J^{\tensor n}
\xymatrix@C50pt{\ar[r]^-{(\eta\,\inj1)\bullet(\eta\,\inj2)}&}
\big( T(J^{\tensor m}+J^{\tensor n}) \big)^{\bullet2}
\xymatrix@C15pt{\ar[r]^-\mu&}
T(J^{\tensor m}+J^{\tensor n})
\end{array}$$
and
$$\begin{array}{rcl}
v_{n,m}
& = &
J^{\tensor(n\cdot m)}
\iso
J^{\tensor n}\bullet J^{\tensor m}
\xymatrix@C50pt{\ar[r]^-{(\eta\,\inj2)\bullet(\eta\,\inj1)}&}
\big( T(J^{\tensor m}+J^{\tensor n}) \big)^{\bullet2}
\xymatrix@C15pt{\ar[r]^-\mu&}
T(J^{\tensor m}+J^{\tensor n})
\end{array}$$
where, for $k,\ell\in\Nat$ and $X\in\Seq$, the isomorphism
$X^{\tensor(k\cdot \ell)}\iso J^{\tensor k}\bullet X^{\tensor\ell}$ is
given by the following composite of canonical isomorphisms:
$$
X^{\tensor(k\cdot\ell)}
\ \iso \
\big(X^{\tensor\ell}\big)^{\tensor k}
\ \iso \
\big(J\bullet X^{\tensor\ell}\big)^{\tensor k}
\ \iso \
J^{\tensor k}\bullet X^{\tensor \ell}
\enspace.
$$

\item
The companion papers~\cite{FioreHur2010} and~\cite{FioreMahmoud2010} consider
MESs for an extension of universal algebra from first to second order,~\ie~to
algebraic languages with variable binding and parameterised metavariables.
This work generalises the semantics of both (first-order) algebraic
theories and of (untyped and simply-typed) lambda calculi.  
\end{enumerate}
\end{examples}

\subsection*{Strong free algebras}

A MES $\mes S=(\cat V,\cat C,\monad T,E)$ is said to admit free algebras
whenever the forgetful functor $U_\mes{S}$ has a left adjoint, so that we have
the following situation:
$$\xymatrix@C=1.5pc@R=2.5pc{
  \catalg{\mes S}\:
  \ar@{^(->}[rr]
  \ar[rd]^-{U_\mes{S}}
  & &
  \cat C^{\monad T}
  \ar[ld]_-{U_{\monad T}}
  \\
  &
  \cat{C}
  \ar@<5pt>@/^10pt/[lu]^-{}
  \ar@<-5pt>@/_10pt/[ru]^-{}
  \ar@<6pt>@{}[lu]|(.5){\leftthreetimes}
  \ar@<-6pt>@{}[ru]|(.5){\rightthreetimes}
  &
}$$
We write $\monad T_{\mes S}$ for the induced \emph{free $\mes S$-algebra
monad} on $\cat C$.

\begin{theorem}\label{MES_StrongMonad}
For a MES $\mes S$ that admits free algebras, the free $\mes S$-algebra monad
is strong.  
\end{theorem}
\proof
By Lemma~\ref{lemma:MES-closed-1} and
Corollary~\ref{corollary:MES-closed-2}, applying
Theorem~\ref{thm:general-free-strong-monad} to the full subcategory
$\catalg{\mes S}$ of $\catalg{T}$ for $T$ the endofunctor underlying the
monad in $\mes S$.
\endproof

\section{Free constructions} 
\label{FreeAlgebrasForMESs}

To establish the wide applicability of Theorem~\ref{MES_StrongMonad}, we give
conditions under which MESs admit free algebras.  The results of this section
follow from the theory developed in~\cite{FioreHur2009}; proofs are
thereby omitted.

\begin{definition}
An object~$A$ of a right-closed $\cat V$-action~$\cat C$ is respectively said
to be \emph{$\kappa$-compact}, for $\kappa$ an infinite limit ordinal, and
\emph{projective} if the functor~$\enrich{\cat C}(A,-):\cat C\rightarrow\cat
V$ respectively preserves colimits of $\kappa$-chains and epimorphisms.
\end{definition}

\begin{definition}
A MES~$(\cat V,\cat C,\monad T, E)$ is called \emph{$\kappa$-finitary}, for
$\kappa$ an infinite limit ordinal, if the category~$\cat C$ is cocomplete,
the endofunctor~$T$ on $\cat C$ preserves colimits of $\kappa$-chains, and the
arity~$A$ of every \mh{T}equation~$\eqnz u v : C\rightarrow TA$ in $E$ is
$\kappa$-compact.  Such a MES is called $\kappa$-inductive if furthermore 
$T$ preserves epimorphisms and the arity~$A$ of every 
\mh{T}equation~$\eqnz u v : C\rightarrow TA$ in $E$ is projective.
\end{definition}

\begin{theorem}
For every \mh\kappa finitary MES~$\mes S=(\cat V,\cat C,\monad T, E)$, 
the embedding ${\catalg{\mes S}\ \rightembedding \cat C^{\monad T}}$ has a
left adjoint, 
the forgetful functor~$U_{\mes S}:\catalg{\mes S}\rightarrow\cat C $ is
monadic,
the category~$\catalg{\mes S}$ is cocomplete, and 
the underlying functor~$T_{\mes S}$ of the monad $\monad T_{\mes S}$
representing $\catalg{\mes S}$ preserves colimits of $\kappa$-chains. 
If, furthermore, $\mes S$ is $\kappa$-finitary then 
$T_{\mes S}$ preserves epimorphisms,
the universal homomorphism from $(TX,\mu_X)$ to its free $\mes S$-algebra
is epimorphic in $\lscat C$, and
free $\mes S$-algebras on $\monad T$-algebras can be constructed in
$\kappa$ steps. 
\end{theorem}

\begin{remark}
The 
theorem above applies to all the examples
of
~\ref{MESsExamples}.  
\end{remark}

\subsection{}

In the case of $\omega$-inductive MESs, the free $\mes S$-algebra 
$(T_{\mes S}X,\tau^{\mes S}_X: TT_{\mes S}X\rightarrow T_\mes{S}X)$ on
$X\in\cat C$ is  constructed as follows:
\begin{equation*}
\label{eqn:qt-con}
\begin{minipage}{.9\textwidth}
\hspace*{-1.2pc}
\xymatrix@C=1.7pc{
\save[]+<-0.2pc,-0.5pc>
*{\mbox{\small${\forall\,(\eqnz u v:C \rightarrow TA) \in E}$}}
\restore
&
T(TX)
\ar[rd]^-{p_0}
\ar@{->>}[r]^-{T(q_0)}
\ar[d]_(.45){\mu_X} 
\ar@{}[rrd]|-{\textsf{po}}
& 
T(TX)_1 
\ar[dr]^-{p_1} 
\ar@{->>}[r]^{T(q_1)} 
\ar@{}[rrd]|-{\textsf{po}}
&
T(TX)_2
\ar@{->>}[r]^{T(q_2)}  
\ar[dr]^-{p_2} 
& 
T(TX)_3 
\ar@{}[r]|-{\cdots\cdots} 
& 
T(T_\mes{S}X) 
\ar@{-->}[d]^-{\tau^\mes{S}_X}  
\\
\save[]+<-2pc,0pc>
*+<6pt,6pt>{\enrich{\cat C}(A,TX)\monact C}
\ar@<3pt>[r]^-{\llrrbrk{u}_{(TX,\mu_X)}}
\ar@<-3pt>[r]_-{\llrrbrk{v}_{(TX,\mu_X)}}
\ar@{}@<10pt>[r]^(0)*{\vdots}
\ar@{}@<-2pt>[r]_(0)*{\vdots}
\restore
{\mbox{}\hspace*{6pc}\mbox{}}
&
TX
\ar@{->>}[r]^-{q_0} 
\ar@{}[r]_-{\s{coeq}} 
& 
(TX)_1 
\ar@{->>}[r]^-{q_1} 
& 
(TX)_2 
\ar@{->>}[r]^-{q_2} 
& 
(TX)_3 
\ar@{}[r]|-{\cdots\cdots} 
\ar@{}@<-7pt>[r]_(1){\textsf{colim}} 
& 
T_\mes{S}X 
}
\end{minipage}
\end{equation*}
where $q_0$ is the universal map that coequalizes every pair $\llrrbrk
u_{(TX,\mu_X)}$ and $\llrrbrk v_{(TX,\mu_X)}$ with  $(\eqnz u v)\in E$; the
parallelograms are pushouts; and $T_{\mes S}X$ is the colimit of the
$\omega$-chain of $q_i$.

Furthermore, when the strong monad $\monad T$ arises from free algebras for a
strong endofunctor $F$ which is \mh\omega cocontinuous and preserves
epimorphisms, the construction simplifies as follows:
\begin{equation*}
\label{eqn:qt-con-simp}
\begin{minipage}{.9\textwidth}
\hspace*{-1.2pc}
\xymatrix@C=1.7pc{
\save[]+<-0.2pc,-0.5pc>
*{\mbox{\small${\forall\,(\eqnz u v:C \rightarrow TA) \in E}$}}
\restore
&
F(TX)
\ar[rd]^-{p_0}
\ar@{->>}[r]^-{F(q_0)}
\ar[d]_(.45){\wh{\mu}_X} 
\ar@{}[rrd]|-{\textsf{po}}
& 
F(TX)_1 
\ar[dr]^-{p_1} 
\ar@{->>}[r]^{F(q_1)} 
\ar@{}[rrd]|-{\textsf{po}}
&
F(TX)_2
\ar@{->>}[r]^{F(q_2)}  
\ar[dr]^-{p_2} 
& 
F(TX)_3 
\ar@{}[r]|-{\cdots\cdots} 
& 
F(T_\mes{S}X) 
\ar@{-->}[d]^-{\wh{\tau}^\mes{S}_X}  
\\
\save[]+<-2pc,0pc>
*+<6pt,6pt>{\enrich{\cat C}(A,TX)\monact C}
\ar@<3pt>[r]^-{\llrrbrk{u}_{(TX,\mu_X)}}
\ar@<-3pt>[r]_-{\llrrbrk{v}_{(TX,\mu_X)}}
\ar@{}@<10pt>[r]^(0)*{\vdots}
\ar@{}@<-2pt>[r]_(0)*{\vdots}
\restore
{\mbox{}\hspace*{6pc}\mbox{}}
&
TX
\ar@{->>}[r]^-{q_0} 
\ar@{}[r]_-{\s{coeq}} 
& 
(TX)_1 
\ar@{->>}[r]^-{q_1} 
& 
(TX)_2 
\ar@{->>}[r]^-{q_2} 
& 
(TX)_3 
\ar@{}[r]|-{\cdots\cdots} 
\ar@{}@<-7pt>[r]_(1){\textsf{colim}} 
& 
T_\mes{S}X 
}
\end{minipage}
\end{equation*}
where $(TX,\wh{\mu}_X)$ and $(T_\mes{S}X,\wh{\tau}^\mes{S}_X)$ are the 
\mh{F}algebras respectively corresponding to the Eilenberg-Moore algebras
$(TX,\mu_X)$ and $(T_\mes{S}X,\tau^\mes{S}_X)$ for the monad $\monad T$.

\section{Equational Metalogic}
\label{EquationalMetaLogic}

The algebraic developments of the paper are put to use in a logical context.
Specifically, as in~\cite{FioreHur2008,Hur2010}, we introduce a
deductive system, here referred to as Equational Metalogic~(EML), for the
formal reasoning about equations in Monadic Equational Systems.  
The envisaged use of EML is to serve as a metalogical framework for the
synthesis of equational logics by instantiating concrete mathematical
models.  This is explained and exemplified in~\cite[Part~II]{FioreHur2011}
and~\cite{FioreHur2010}.

\subsection*{Equational Metalogic}
\label{EquationalMetalogicSubSection}

The \emph{Equational Metalogic} associated to a MES $
(\cat V,\cat C,\monad T,\Ax)$ 
consists of inference rules that inductively define the derivable
equational consequences
\[
\eqn{\Ax}{u}{v}:C\rightarrow TA
\enspace,
\]
for $u$ and $v$ Kleisli maps of arity~$A$ and coarity~$C$, that
follow from the equational presentation~$E$.  

EML has been synthesised from the model theory, in that each inference rule
reflects a model-theoretic property of equational satisfaction arising from
the algebraic structure of the semantic interpretation.
The inference rules of EML, besides those of equality and axioms, consist
of congruence rules for composition and monoidal action, and a rule for
the local character~(see~\eg~\cite[page~316]{MacLaneMoerdijk1992}) of
derivability.  Formally, these are as follows.
\begin{enumerate}
\item
Equality rules.\\
\[
\myproof{
\AxiomC{$\phantom{\Ax}$}
\LeftLabel{$\s{Ref}$}
\UnaryInfC{$\eqn \Ax u u : C\rightarrow TA$}
}
\qquad\qquad\qquad
\myproof{
\AxiomC{$\eqn \Ax u v : C\rightarrow TA$}
\LeftLabel{$\s{Sym}$}
\UnaryInfC{$\eqn \Ax v u : C\rightarrow TA$}
}
\]

\[
\myproof{
\AxiomC{
  $\eqn \Ax u v: C\rightarrow TA$
  \qquad 
  $\eqn \Ax v w: C\rightarrow TA$}
\LeftLabel{$\s{Trans}$}
\UnaryInfC{
  $\eqn \Ax u w: C\rightarrow TA$}
}
\]\\[-2.5mm]

\item
Axioms.\\
\[
\myproof{
\LeftLabel{$\s{Axiom}$}
\AxiomC{$(\eqnz u v: C\rightarrow TA)\in \Ax$}
\UnaryInfC{$
\eqn \Ax u v : C\rightarrow TA
$}
}
\]\\[-2.5mm]

\item
Congruence of composition.\\
\[
\myproof{
\AxiomC{$
\eqn{\Ax}{u_1}{v_1}:C \rightarrow TB
\qquad
\eqn{\Ax}{u_2}{v_2}:B \rightarrow TA
$}
\LeftLabel{$\s{Comp}$}
\UnaryInfC{$
\eqn{\Ax}{u_1\subst{u_2}\,}{\, v_1\subst{v_2}}:C\rightarrow TA$}
}
\]
where $w_1\subst{w_2}$ denotes the Kleisli composite
$\!\!\xymatrix@R=0pt@C=7.5pt{ 
  C\ar[rr]^-{w_1} && TB \ar[rrr]^-{T(w_2)} &&& T(TA) \ar[rr]^-{\mu_A} &&
  TA }$.\\

\item
Congruence of monoidal action.\\
\[
\myproof{
\AxiomC{$
\eqn{\Ax}{u}{v}: C\rightarrow TA
$}
\LeftLabel{$\s{Ext}$}
\RightLabel{$(V\in\cat{V})$}
\UnaryInfC{$
\eqn{\Ax}{\tensorext{V}{u} \,}{\, \tensorext{V}{v}}
  : V\monact C\rightarrow T(V\monact A)
$}
}
\]
where 
$\tensorext{V}{w}$ denotes the composite
$\!\!\xymatrix@R=0pt@C=30pt{
V\monact C \ar[r]^-{V\monact w} 
& 
V\monact TA \ar[r]^-{\st_{V,A}} 
& 
T(V\monact A)
}$.\\

\item
Local character.\\
\[
\myproof{
\AxiomC{$\eqn{\Ax}{ u\comp e_i \,}{\, v\comp e_i}
    : C_i\rightarrow TA \quad (i\in I)$}
\LeftLabel{$\s{Local}$}
\RightLabel{\big($\setof{e_i:C_i\rightarrow C}_{i\in I}$ jointly epi\big)}
\UnaryInfC{$\eqn \Ax u  v : C\rightarrow TA$}
}
\]\\
(Recall that a family of maps $\setof{e_i:C_i\rightarrow C}_{i\in I}$ is said
to be \emph{jointly epi} if, for any $f,g:C\rightarrow X$ such that
$\forall_{i\in I}\ 
 {f\comp e_i \, = \, g\comp e_i: C_i\rightarrow X}$,
it follows that $f\,=\,g$.)\\
\end{enumerate}

\begin{remark}
In 
the presence of coproducts and under the rule~$\s{Ref}$, the
rules~$\s{Comp}$ and~$\s{Local}$ are inter-derivable with the
rules\\[-1mm]
\[\myproof{
\AxiomC{$
\eqn{\Ax}{u}{v}: C\rightarrow T\big(\coprod_{i\in I}B_i\big)
\qquad
\eqn{\Ax}{u_i}{v_i}: B_i\rightarrow TA \enspace (i\in I)$}
\LeftLabel{$\s{Comp}_\amalg$}
\UnaryInfC{$\eqn\Ax{u\subst{[u_i]_{i\in I}}}{v\subst{[v_i]_{i\in I}}}
:C\rightarrow TA$}
}\]\\
and\\[-2mm]
\[\myproof{
\AxiomC{$\eqn{\Ax}{u\comp e}{v \comp e}: C'\rightarrow TA$}
\LeftLabel{$\s{Local}_1$}
\RightLabel{($e:C'\epirightarrow C$ epi)}
\UnaryInfC{$\eqn \Ax u  v : C\rightarrow TA$}
}\]\\[-6mm]
\end{remark}

\subsection*{Soundness}

The minimal requirement for a deductive system to be of interest is that
of \emph{soundness};~\ie~that derivability entails validity.

We show that that EML is sound for the model theory of MESs.

\begin{theorem}\label{EMLsoundness}
For a MES~$\mes S=(\cat V,\cat C,\monad T,\Ax)$,
\begin{center}
if $\eqn \Ax u v : C\rightarrow TA$ is derivable in EML then
$\catalg{\mes{S}}\models \eqnz u v : C\rightarrow TA$
\enspace.
\end{center}
\end{theorem}
\proof
One shows the soundness of each rule of EML;~\ie~that every 
\mh{\mes S}algebra satisfying the premises of an EML rule also satisfies its
conclusion.  

The soundness of the rules $\s{Ref}$, $\s{Sym}$, $\s{Trans}$, and
$\s{Axiom}$ is trivial.

For the rest of the proof, let $\ol f: Z\rightarrow\homact{V,Y}$ denote
the transpose of $f:V\monact Z\rightarrow Y$; so that $\ol{\llrrbrk
t_{(X,s)}} = \stmonmor(s)_A\comp t: C\rightarrow\homact{\enrich{\cat
C}(A,X),X}$ for all $t:C\rightarrow TA$. 

The soundness of the rule~$\s{Comp}$ is a consequence of the functoriality
of ${\Klfun{\stmonmor(s)}:\Kl{\cat C}{\monad T}\rightarrow\Kl{\cat
C}{\DDmon X}}$, see Remark~\ref{SemanticsKleisliFunctorRemark}, from which we
have that 
$$
\ol{\llrrbrk{w_1\subst{w_2}}_{(X,s)}}
= 
\ol{\llrrbrk{w_2}_{(X,s)}} 
\,\comp_{\Kl{\cat C}{\DDmon X}}
\ol{\llrrbrk{w_1}_{(X,s)}}
: C\rightarrow\homact{\enrich{\cat C}(A,X),X}
$$
for all $w_1:C\rightarrow TB$ and $w_2:B\rightarrow TA$ in $\cat C$.

The soundness of the rule~$\s{Ext}$ is a consequence of the commutativity
of~(\ref{SemanticsKleisliFunctor}), from which we have that 
$$
\ol{\llrrbrk{\tensorext V t}_{(X,s)}}
=
\DDst X_{V,A}
\comp
(V\monact\ol{\llrrbrk t_{(X,s)}})
: C\rightarrow\homact{\enrich{\cat C}(A,X),X}
$$
for all $t:C\rightarrow TA$ in $\cat C$.

Finally, the soundness of the rule~$\s{Local}$ is a consequence of the
fact that $\ol{\llrrbrk{t\comp e}_{(X,s)}} = \ol{\llrrbrk t_{(X,s)}}\comp
e$.  
\endproof

\section{Internal strong completeness}
\label{InternalCompletenessSection}

The completeness of EML, \emph{i.e.}~the converse to the soundness
theorem, 
cannot be established at the abstract level of generality that we are working
in.  We do however have an internal form of strong completeness for Monadic
Equational Systems admitting free algebras.  The main development of this
section is to state and prove this result.

The 
internal strong completeness theorem in conjunction with the construction
of free algebras provides a main mathematical tool for establishing the
completeness of concrete instantiations of EML,
see~\cite[Part~II]{FioreHur2011} and~\cite{FioreHur2010}.

\begin{notation}
For a MES $\mes S=(\cat V,\cat C,\monad T,E)$ admitting free algebras,
write $(T_{\mes S}X,\tau^{\mes S}_X: TT_{\mes S}X\rightarrow T_\mes{S}X)$
for the free \mh{\mes S}algebra on an object~$X\in\cat C$.  
\end{notation}
Then, the family~$\tau^\mes{S}=\setof{\tau^\mes{S}_X}_{X\in\cat C}$ yields
a natural transformation~$\tau^\mes{S}:TT_\mes{S}\rightarrow
T_\mes{S}$.  

\subsection*{Quotient maps}

Let $\mes S$ be a MES admitting free algebras.  The universal property of
free \mh{\monad T}algebras induces a family of
morphisms~$\qt^\mes{S}=\setof{\qt^\mes{S}_X:TX\rightarrow T_{\mes
S}X}_{X\in\cat C}$, referred to as the \emph{quotient maps} of $\mes S$,
defined as the unique homomorphic extensions
$(TX,\mu_X)\rightarrow(T_{\mes S}X,\tau^{\mes S}_X)$ of $\eta^{\mes S}_X$;
\ie~the unique maps such that the diagram 
\begin{equation}
\label{eqn:quotient-map}
\begin{minipage}{\textwidth}
\xymatrix@C=4pc{
  TTX
  \ar[r]^-{T(\qt^\mes{S}_X)}
  \ar[d]_-{\mu_X}
  &
  TT_{\mes S}X
  \ar[d]^-{\tau^{\mes S}_X}
  \\
  TX
  \ar[r]^-{\exists!\, \qt^\mes{S}_X}
  &
  T_{\mes S}X
  \\
  X
  \ar[u]^-{\eta_X}
  \ar[ru]_-{\eta^{\mes S}_X}
}
\end{minipage}
\end{equation}
commutes.  
As a general $2$-categorical fact, 
the family~$\setof{\qt^\mes{S}_X}_{X\in\cat C}$ yields a monad
morphism~$\qt^{\mes S}: \monad T\rightarrow\monad T_{\mes S}$.
By Theorem~\ref{MES_StrongMonad}, the free \mh{\mes S} algebra monad~$\monad
T_\mes{S}$ is strong, and we proceed to show that so is the monad
morphism~$\qt^\mes{S}$. 
\begin{theorem}\label{StrongFunctorMorphismTheorem}
For a MES~$\mes S = (\cat V,\cat C,\monad T,E)$, the monad
morphism~$\qt^{\mes S}:T\rightarrow T_{\mes S}$ is strong.
\end{theorem}
\proof
The result follows from
Theorem~\ref{thm:general-free-strong-monad}\,(\ref{thm:general-free-strong-monad-1})
applied in the case $\cat A = \cat C^\monad{T}$ by virtue of
Lemma~\ref{lemma:MES-closed-1}, 
showing that the composites
$$
\xymatrix@C=4pc{
  V\monact TX
  \ar[r
]^-{\st_{V,X}}
  &
  T(V\monact X)
  \ar[r]^-{\qt^{\mes S}_{V\monact X}}
  &
  {T}_{\mes S}(V\monact X)
}
$$
and
$$
\xymatrix@C=4pc{
  V\monact TX
  \ar[r]^-{V\monact \qt^{\mes S}_X}
  &
  V\monact {T}_{\mes S}X
  \ar[r
]^-{{\st}^{\mes S}_{V,X}}
  &
  {T}_{\mes S}(V\monact X)
}
$$
are the unique extension of 
$\eta^{\mes S}_{V\monact X}:V\monact X\rightarrow T_{\mes S}(V\monact X)$.
\endproof

\begin{proposition}\label{QuotientAsInterpretation}
For a MES $\mes S$ admitting free algebras, the quotient maps 
$\qt^{\mes S}_X$ factor as 
$$\xymatrix@C=20pt{
  TX \ar[rr]^-{(\lambdaact_{TX})^{-1}} &&
  I\monact TX \ar[rr]^-{n_X\monact TX} &&
  \enrich{\cat C}(X,T_{\mes S}X) \monact TX 
  \ar[rrr]^-{\intmap(\tau^{\mes S}_X)} 
  &&& X
}$$
where $n_X$ is the transpose of 
$\xymatrix{I\monact X\ar[r]^-{\lambdaact_X} & X \ar[r]^-{\eta^{\mes S}_X}
  & T_{\mes S}X}$.  
\end{proposition}
\proof
Noting that $\qt^{\mes S}_X$ factors as $\tau^{\mes S}_X \comp
T(\eta^{\mes S}_X)$, since this map is also an homomorphic extension
$(TX,\mu_X)\rightarrow(T_{\mes S}X,\tau^{\mes S}_X)$ of $\eta^{\mes S}_X$,
one calculates as follows
$$\xymatrix@C=70pt{
  &
  \ar[ld]_-{(\lambdaact_{TX})^{-1}}
  TX
  \ar[r]^-{\qt^\mes{S}_X}
  \ar[rd]_-{T(\eta^{\mes S}_X)}
  \ar[d]^-{T((\lambdaact_X)^{-1})}
  &
  T_\mes{S}X
  \\
  I\monact TX
  \ar[r]^-{\st_{I,X}}
  \ar[rd]_-{n_X\monact TX}
  &
  T(I\monact X)
  \ar[rd]^-{T(n_X\monact X)}
  &
  TT_\mes{S}X
  \ar[u]_-{\tau^\mes{S}_X}
  \\
  &
  \enrich{\cat C}(X,T_\mes{S}X)\monact TX
  \ar[r]_-{\st_{\enrich{\cat C}(X,T_\mes{S}X),X}}
  &
  T(\enrich{\cat C}(X,T_\mes{S}X)\monact X)
  \ar[u]_-{T(\eneval{X}{T_\mes{S}X})}
}
$$
\endproof

\begin{definition}
Let $\mes S$ be a MES admitting free algebras.  For an \mh{\mes S}algebra
${s:TX\rightarrow X}$, let $\wt s:T_\mes{S}X \rightarrow X$ be the unique
homomorphic extension $(T_{\mes S}X,\tau^{\mes S}_X)\rightarrow(X,s)$ of
the identity on $X$, so that 
$$
\xymatrix@C=4pc{
  TT_\mes{S}X
  \ar[r]^-{T(\wt s)}
  \ar[d]_-{\tau^\mes{S}_X}
  &
  TX
  \ar[d]^-{s}
  \\
  T_\mes{S}X
  \ar[r]^-{\exists!\, \wt s}
  &
  X
  \\
  X
  \ar[u]^-{\eta^{\mes S}_X}
  \ar[ru]_-{\id_X}
}
$$
\end{definition}

\begin{proposition}\label{QuotientFactoring}
For a MES $\mes S$ admitting free algebras, every \mh{\mes S}algebra
$s:TX\rightarrow X$ factors as the composite
$$\xymatrix{
TX \ar[r]^-{\qt^{\mes S}_X} & T_{\mes S}X \ar[r]^-{\wt s} & X
\enspace.
}$$
\end{proposition}
\proof
As 
both morphisms are the unique homomorphic extension
$(TX,\mu_X)\rightarrow(X,s)$ of $\id_X$.
\endproof

\subsection*{Internal strong completeness}

The main result of the section~\cite{FioreHur2008,Hur2010} follows.

\begin{theorem}
\label{thm:int-comp}
For a MES~$\mes S = (\cat V,\cat C,\monad T,E)$ 
admitting free algebras,
the following are equivalent. 
\begin{enumerate}
\item\label{thm:int-comp-1}
$\catalg{\mes S} \models \eqnz u v:C\rightarrow TA$.

\item\label{thm:int-comp-2}
$(T_\mes{S} A,\tau^\mes{S}_A) \models \eqnz u v:C\rightarrow TA$.

\item\label{thm:int-comp-3}
$\qt^\mes{S}_A\comp u \;=\; \qt^\mes{S}_A\comp v:
C\rightarrow T_\mes{S} A$.
\end{enumerate}
\end{theorem}
Here, the equivalence of the first two statements is an internal form of
so-called \emph{strong completeness}, stating that an equation is satisfied by
all models if and only if it is satisfied by a freely generated one.  
\proof
(\ref{thm:int-comp-1}) $\Rightarrow$ (\ref{thm:int-comp-2}).  
Holds vacuously.

(\ref{thm:int-comp-2}) $\Rightarrow$ (\ref{thm:int-comp-3}).
Because 
$\qt^{\mes S}_A\comp t
=
\llrrbrk t_{(T_{\mes S}A,\tau^{\mes S}_A)}
\comp
(n_A\monact C)
\comp
(\lambdaact_C)^{-1}$
for all $t:C\rightarrow TA$, 
as follows from Proposition~\ref{QuotientAsInterpretation}.

(\ref{thm:int-comp-3}) $\Rightarrow$ (\ref{thm:int-comp-1}). 
Because 
$\llrrbrk t_{(X,s)}
=
\llrrbrk{\qt^{\mes S}_A\comp t}_{(X,\wt s)}$
for all $t:C\rightarrow TA$, 
as follows from the identity
$$\begin{array}{rcll}
\intmap(s)_A 
  & = & \intmap(\wt s \comp \qt^{\mes S}_X)_A
  & \mbox{, by Proposition~\ref{QuotientFactoring}}
\\[2mm]
  & = & 
  \intmap(\wt s)_A \comp (\enrich{\cat C}(A,X)\monact\qt^{\mes S}_A)
  & \mbox{, by Proposition~\ref{StMonMorInt}}
\end{array}$$
\endproof

%

\refs


\bibitem[Ad\'amek, H\'ebert and Sousa~(2007)]{AdamekHebertSousa2007}
J.\,Ad\'amek, M.\,H\'ebert and L.\,Sousa.
\newblock A logic of injectivity. 
\newblock \emph{Journal of Homotopy and Related Topics}, 2:13--47, 2007.

\bibitem[Ad\'amek and Rosick\'y~(1994)]{AdamekRosicky1994}
J.\,Ad\'amek and J.\,Rosick\'y.
\newblock \emph{Locally presentable and accessible categories}.
\newblock Cambridge University Press, 1994.

\bibitem[Ad\'amek, Rosick\'y and Vitale~(2010)]{AdamekRosickyVitale2010}
J.\,Ad\'amek, J.\,Rosick\'y and E.\,Vitale.
\newblock \emph{Algebraic theories: A categorical introduction to general
  algebra}.
\newblock Cambridge University Press, 2010.

\bibitem[Ad\'amek, Sobral and Sousa~(2009)]{AdamekSobralSousa2009}
J.\,Ad\'amek, M.\,Sobral and L.\,Sousa.
\newblock A logic of implications in algebra and coalgebra.
\newblock \emph{Algebra Universalis}, 61:313--337, 2009.

\bibitem[Birkhoff~(1935)]{Birkhoff1935}
G.\,Birkhoff. 
\newblock On the structure of abstract algebras. 
\newblock \emph{Mathematical Proceedings of the Cambridge Philosophical
  Society}, 31(4):433--454, 1935.

\bibitem[Borceux and Day~(1980)]{BorceuxDay1980}
F.\,Borceux and B.\,Day.
\newblock Universal algebra in a closed category.
\newblock \emph{Journal of Pure and Applied Algebra}, 16:133--147, 1980.

\bibitem[Burroni~(1971)]{Burroni1971}
A.\,Burroni. 
\newblock $T$-cat\'egories (Cat\'egories dans un triple).
\newblock \emph{Cahiers de Topologie et G\'eom\'etrie Diff\'erentielle}, 
  XII(3):245--321, 1971.

\bibitem[Cohn~(1965)]{Cohn1965}
P.\,Cohn.
\newblock \emph{Universal algebra}.
\newblock Harper \& Row, 1965.

\bibitem[Day~(1970)]{Day1970}
B.\,Day.
\newblock On closed categories of functors.
\newblock \emph{Reports of the Midwest Category Seminar IV}, volume 137 of
  Lecture Notes in Mathematics, pages~1--38.
\newblock Springer-Verlag, 1970.

\bibitem[Ehresmann~(1968)]{Ehresmann1968}
C.\,Ehresmann. 
\newblock Esquisses et types des structures alg\'ebriques. 
\newblock \emph{Bul.\ Inst.\ Polit.\ Iasi}, XIV, 1968.

\bibitem[Fiore, Plotkin and Turi~(1999)]{FiorePlotkinTuri1999}
M.\,Fiore, G.\,Plotkin and D.\,Turi.
\newblock Abstract syntax and variable binding.
\newblock In \emph{Proceedings of the 14th Annual IEEE Symposium on Logic in
  Computer Science~(LICS'99)}, pages~193--202.
\newblock IEEE, Computer Society Press, 1999.

\bibitem[Fiore~(2008)]{Fiore2008}
M.\,Fiore. 
\newblock Second-order and dependently-sorted abstract syntax.
\newblock In \emph{Logic in Computer Science Conf.}~(LICS'08),
  pages~57--68.
\newblock IEEE, Computer Society Press, 2008.

\bibitem[Fiore, Gambino, Hyland, and Winskel~(2008)]
  {FioreGambinoHylandWinskel2008}
M.\,Fiore, N.\,Gambino, M.\,Hyland, and G.\,Winskel.   
\newblock The cartesian closed bicategory of generalised species of
  structures.  
\newblock \emph{J.\ London Math.\ Soc.}, 77:203-220, 2008. 

\bibitem[Fiore and Hur~(2008)]{FioreHur2008}
M.\,Fiore and C.-K.\,Hur. 
\newblock Term equational systems and logics. 
\newblock In \emph{Proceedings of the Twenty-Fourth Conference on the
Mathematical Foundations of Programming Semantics}~(MFPS'08), volume 218
of Electronic Notes in Theoretical Computer Science, pages~171--192.  
\newblock Elsevier, 2008.

\bibitem[Fiore and Hur~(2009)]{FioreHur2009}
M.\,Fiore and C.-K.\,Hur.
\newblock On the construction of free algebras for equational systems.
\newblock \emph{Theoretical Computer Science}, 410(18):1704--1729, 2009.

\bibitem[Fiore and Hur~(2010)]{FioreHur2010}
M.\,Fiore and C.-K.\,Hur. 
\newblock Second-order equational logic. 
\newblock In \emph{Proceedings of the 19th EACSL Annual Conference on
Computer Science Logic}~(CSL 2010), volume 6247 of Lecture Notes in
Computer Science, pages~320--335. 
\newblock Springer-Verlag, 2010.

\bibitem[Fiore and Hur~(2011)]{FioreHur2011}
M.\,Fiore and C.-K.\,Hur. 
\newblock On the mathematical synthesis of equational logics.
\newblock \emph{Logical Methods in Computer Science}, 7(3:12), 2011.

\bibitem[Fiore and Mahmoud~(2010)]{FioreMahmoud2010}
M.\,Fiore and O.\,Mahmoud. 
\newblock Second-order algebraic theories. 
\newblock In \emph{Proceedings of the 35th International Symposium on
Mathematical Foundations of Computer Science}~(MFCS 2010), volume 6281 of
Lecture Notes in Computer Science, pages~368--380. 
\newblock Springer-Verlag, 2010.

\bibitem[Gabbay and Pitts~(2001)]{GabbayPitts2001}
M.\,J.\,Gabbay and A.\,Pitts.
\newblock A new approach to abstract syntax with variable binding.
\newblock \emph{Formal Aspects of Computing}, 13:341--363, 2001.


\bibitem[Hamana~(2004)]{Hamana2004}
M.\,Hamana.
\newblock Free {$\Sigma$}-monoids: A higher-order syntax with metavariables.
\newblock In \emph{Second ASIAN Symposium on Programming Languages and Systems
  (APLAS~2004)}, volume 3302 of Lecture Notes in Computer Science,
  pages~348--363, 2004.

\bibitem[Hur~(2010)]{Hur2010}
C.-K.\,Hur. 
\newblock \emph{Categorical equational systems: Algebraic models and equational
  reasoning}.
\newblock PhD~thesis, Computer Laboratory, University of Cambridge, 2010.

\bibitem[Hyland and Power~(2007)]{HylandPower2007}
M.\,Hyland and A.\,J.\,Power.
\newblock The category theoretic understanding of universal algebra:
  Lawvere theories and monads.
\newblock \emph{Electronic Notes in Theoretical Computer Science},
  172:437--458, 2007.

\bibitem[Im and Kelly~(1986)]{ImKelly1986}
G.\,Im and G.\,M.\,Kelly.
\newblock A universal property of the convolution monoidal structure.
\newblock \emph{Journal of Pure and Applied Algebra}, 43:75--88, 1986.

\bibitem[Janelidze and Kelly~(2001)]{JanelidzeKelly01}
G.\,Janelidze and G.\,M.\,Kelly.
\newblock A note on actions of a monoidal category.
\newblock \emph{Theory and Applications of Categories}, 9(4):61--91, 2001.
 
\bibitem[Joyal~(1981)]{Joyal1981}
A.\,Joyal.
\newblock Une theorie combinatoire des s\'eries formelles.
\newblock \emph{Advances in Mathematics}, 42:1--82, 1981.

\bibitem[Kelly~(1972)]{Kelly1972}
G.\,M.\,Kelly.
\newblock On the operads of J.\,P.\,May.
\newblock Unpublished manuscript, 1972.
\newblock (Reprints in \emph{Theory and Applications of Categories}, No.~13,
pages~1--13, 2005.)

\bibitem[Kelly and Power~(1993)]{KellyPower1993}
G.\,M.\,Kelly and A.\,J.\,Power. 
\newblock Adjunctions whose counits are coequalizers, and presentations of
  finitary enriched monads. 
\newblock \emph{Journal of Pure and Applied Algebra}, 89:163--179, 1993.

\bibitem[Kock~(1970a)]{Kock1970a}
A.\,Kock.
\newblock On double dualization monads.
\newblock \emph{Math.\ Scand.}, 27:151--165, 1970.

\bibitem[Kock~(1970b)]{Kock1970b}
A.\,Kock.
\newblock Monads on symmetric monoidal closed categories.
\newblock \emph{Archiv der Mathematik}, XXI:1--10, 1970.

\bibitem[Kock~(1972)]{Kock1972}
A.\,Kock.
\newblock Strong functors and monoidal monads.
\newblock \emph{Archiv der Math.}, XXIII:113--120, 1972.

\bibitem[Kock~(2012)]{Kock2012}
A.\,Kock.
\newblock Commutative monads as a theory of distributions.
\newblock \emph{Theory and Applications of Categories}, 26(4):97--131, 2012.

\bibitem[Lack and Power~(2009)]{LackPower2009}
S.\,Lack and A.\,J.\,Power.
\newblock Gabriel-Ulmer duality and Lawvere theories enriched over a
  general base.
\newblock \emph{J.\ Funct.\ Programming}, 19(3--4):265--286, 2009.

\bibitem[Lack and Rosick\'y~(2011)]{LackRosicky2011}
S.\,Lack and J.\,Rosick\'y.
\newblock Notions of Lawvere theory.
\newblock \emph{Appl.\ Categor.\ Struct.}, 19:363--391, 2011.

\bibitem[Lawvere~(1963)]{Lawvere1963}
F.\,W.\,Lawvere. 
\newblock \emph{Functorial semantics of algebraic theories}. 
\newblock PhD thesis, Columbia University, 1963.
\newblock (Republished in \emph{Reprints in Theory and Applications of
  Categories}, 5:1--121, 2004.)

\bibitem[Linton~(1966)]{Linton1966}
F.\,Linton. 
\newblock Some aspects of equational theories. 
\newblock \emph{Proc.\ Conf.\ on Categorical Algebra at La Jolla},
  pages~84--95, 1966.

\bibitem[Mac~Lane~(1965)]{MacLane1965}
S.\,Mac Lane.
\newblock Categorical algebra.
\newblock \emph{Bulletin of the American Mathematical Society} 71(1):40--106,
  1965.

\bibitem[Mac~Lane~(1997)]{MacLaneCWM}
S.\,Mac Lane.
\newblock \emph{Categories for the working mathematician} (second edition).
\newblock Springer-Verlag, 1997.

\bibitem[Mac~Lane and Moerdijk~(1992)]{MacLaneMoerdijk1992}
S.\,Mac Lane and I.\,Moerdijk. 
\newblock \emph{Sheaves in geometry and logic}. 
\newblock Springer-Verlag, 1992.  

\bibitem[MacDonald and Sobral~(2004)]{MacDonaldSobral2004}
J.\,MacDonald and M.\,Sobral.
\newblock Aspects of monads.
\newblock Chapter~5 of \emph{Categorical Foundations: Special Topics in
  Order, Topology, Algebra, and Sheaf Theory}. 
\newblock Cambridge University Press, 2004.

\bibitem[Pareigis~(1977)]{Pareigis1977}
B.\,Pareigis.
\newblock Non-additive ring and module theory II.  \mh{\cat C}categories,
\mh{\cat C}functors and \mh{\cat C}morphisms.
\newblock \emph{Publicationes Mathematicae}, 25:351--361, 1977. 

\bibitem[Pedicchio and Rovatti~(2004)]{PedicchioRovatti2004}
M.\,C.\,Pedicchio and F.\,Rovatti.
\newblock Algebraic categories.
\newblock Chapter~6 of \emph{Categorical Foundations: Special Topics in
  Order, Topology, Algebra, and Sheaf Theory}.
\newblock Cambridge University Press, 2004.

\bibitem[Power~(1999)]{Power1999}
A.\,J.\,Power.
\newblock Enriched Lawvere theories. 
\newblock \emph{Theory and Applications of Categories}, 6:83--93, 1999.

\bibitem[Robinson~(2002)]{Robinson2002}
E.\,Robinson. 
\newblock Variations on algebra: Monadicity and generalisations of equational
  theories.  
\newblock \emph{Formal Aspects of Computing}, 13(3--5):308--326, 2002.

\bibitem[Ro\c{s}u~(2001)]{Rosu2001}
G.\,Ro\c{s}u.
\newblock Complete categorical equational deduction.
\newblock In \emph{Computer Science Logic}, volume 2142 of Lecture Notes
  in Computer Science, pages~528--538.  
\newblock Springer-Verlag, 2001.

\bibitem[S{\l}ominski~(1959)]{Slominski1959}
J.\,S{\l}ominski.
\newblock The theory of abstract algebras with infinitary operations.
\newblock \emph{Rozprawy Mat.}~18, 1959.

\bibitem[Tanaka~(2000)]{Tanaka2000}
M.\,Tanaka.
\newblock Abstract syntax and variable binding for linear binders.
\newblock In \emph{Proceedings of the 25th International Symposium on
  Mathematical Foundations of Computer Science~(MFCS'00)}, volume 1893 of
  Lecture Notes in Computer Science, pages~670--679.  
\newblock Springer-Verlag, 2000.

\bibitem[Wraith~(1975)]{Wraith1975}
G.\,Wraith.
\newblock Algebraic theories. 
\newblock Lecture Notes Series No.~22. 
\newblock Matematisk Institut, Aarhus Universitet, 1975.

\endrefs

\end{document}